\mag=\magstep1 
\input epsf 
\documentclass[10pt,a4paper]{amsart} 
\usepackage{amssymb,latexsym} 

\def\Cal{\mathcal} 
 
\def\frak{\mathfrak}

\def\ot{\leftarrow} 
\def\<<{\langle } 
\def\>>{\rangle }

\numberwithin{equation}{section} 
\newtheorem{theorem}{Theorem}[section] 
\newtheorem{proposition}[theorem]{Proposition} 
\newtheorem{corollary}[theorem]{Corollary} 
\newtheorem{definition}[theorem]{Definition} 
\newtheorem{conjecture}[theorem]{Conjecture}

\newtheorem{example}[theorem]{Example}

\def\Ker{\operatorname{Ker}}

\def\<{\langle} 
\def\>{\rangle}

\begin{document} 
 
\title{The Artin-Schreier DGA and the $\bold F_p$-fundamental 
group of an $\bold F_p$ scheme} 
 
\author{Tomohide Terasoma}

\maketitle 
 
\makeatletter 
\renewcommand{\@evenhead}{\tiny \thepage \hfill 
\hfill}

\renewcommand{\@oddhead}{\tiny \hfill Tomohide Terasoma
 \hfill \thepage} 
\tableofcontents
\section{Introduction and convention}
\subsection{Introduction}
N.Katz proved the following theorem.
\begin{theorem}
\label{Katz theorem}
Let $X$ be a connected separated scheme of 
finite type over $\bold F_p$ and $F:X\to X$
the absolute Frobenius map.
Then there exists an equivalence of categories between
\begin{enumerate}
\item
the category $(LS_{\bold F_p}/X_{et})$
of smooth etale $\bold F_p$ sheaves on $X$ of finite rank, and
\item
the category $(UR_{\bold F_p}/X)$ of pairs $(\Cal F, \varphi)$ consisting of
\begin{enumerate}
\item
a locally free $\Cal O_X$ module $\Cal F$ of finite rank and
\item
an isomorphism $\varphi:F^*\Cal F \to \Cal F$ of locally free
$\Cal O_X$ modules.
\end{enumerate}

\end{enumerate}
\end{theorem}

In this theorem, the the functor from $(UR_{\bold F_p}/X)$ to $(LS_{\bold F_p}/X_{et})$
is given as follows. Let $h:W\to X$ be an etale morphism and $F_X:X\to X$
and $F_W:W \to W$ the absolut frobenius maps. Then the relative
frobenius map $\Phi$ is defined by the following commutative diagram:

\setlength{\unitlength}{0.65mm}
\begin{picture}(200,40)(-14,3)
\put(50,2){\vector(1,0){30}}
\put(65,20){\vector(1,0){15}}
\put(47,16){\vector(0,-1){10}}
\put(84,16){\vector(0,-1){10}}
\put(86,10){$h$}
\put(45,0){$X$}
\put(60,4){$F_X$}
\put(82,0){$X$}
\put(38,18){$W\times_{X,F}X$}
\put(82,18){$W$}
\put(20,32){$W$}
\put(28,34){\vector(4,-1){52}}
\put(26,30){\vector(2,-3){18}}
\put(27,13){$h$}
\put(60,29){$F_W$}
\put(27,32){\vector(1,-1){10}}

\put(31,27){$\Phi$}

\end{picture}
\vskip 0.2in
\noindent
The map $\Phi :W\to W\times_{X,F}X$ induces a map 
$\varphi^*:F_*\Cal F(U)=\Cal F(U\times_{X,F}X)\to \Cal F(U)$,
and we have the composite map $\Phi^*\circ{\ ^a\varphi}:\Cal F(U)\to \Cal F(U)$ 
of $\Phi^*$ and the adjoint $^a\varphi:\Cal F(U)
\to F_*\Cal F(U)$ of the map $\varphi$.
Then the functor $W\mapsto Ker(\Phi^*\circ{\ ^a\varphi}-1)\subset \Cal
F(W)$ becomes a smooth $\bold F_p$ etale sheaf of finte rank on $X_{et}$.
The construction of the inverse functor is based on the etale descent theory of locally free 
$\Cal O_X$ sheaves of finite rank.

An object in the category $(UR_{\bold F_p}/X)$ is sometimes called 
a unit root $\bold F_p$-F-crystal on
$X$.
A smooth etale $\bold F_p$ sheaf $V$ is called unipotent if there
is a filtration $\{F^iV\}$ such that the associated graded
sheaf $Gr^i_F(V)$ is a constant sheaf on
$X$. The full subcategory of unipotent $\bold F_p$ etale local systems
is denoted by $(NLS_{\bold F_p}/X_{et})$. 
The full subcategory of $(UR_{\bold F_p}/X)$ corresponding to
$(NLS_{\bold F_p}/X_{et})$ is denoted by $(NUR_{\bold F_p}/X)$.
For the moment,
we assume that there exists an $\bold F_p$-valued point $x$ in $X$. By
taking a geometric point $\bar x$ over $x$, we have the fundamental group
of the category $(NLS_{\bold F_p}/X_{et})$ with respect to the base point
$\bar x$, which is called the $\bold F_p$-completion 
$\pi_1(X,\bar x)^{\bold F_p}$
of the etale fundamental group of $X$.
Note that $\pi_1(X,\bar x)^{\bold F_p}$ is isomorphic to the
pro-p completion of $\pi_1(X,\bar x)$. (See Section \ref{appendix}.)

Let $\bold F_p[[\pi_1(X,\bar x)]]^*$ be the topological dual 
of the completion of $\bold F_p[\pi_1(X,\bar x)]$
with respect to the topology defined by the augmentation ideal
$I$.
Then it is equipped with a structure of Hopf algebra.
The $\bold F_p$-completion 
$\pi_1(X,\bar x)^{\bold F_p}$
is isomorphic to the group of group like elements in
the Hopf algebra $\bold F_p[[\pi_1(X,\bar x)]]^*$.
In this paper, we show that a Hopf algebra 
$\bold F_p[[\pi_1(X,\bar x)]]^*$ is isomorphic
to the cohomology of the bar complex of the Artin-Schreier
DGA of $X$ in Section \ref{AS DGA}. 

The bar construction is a
standard construction to get a coalgebra
out of an associative differential graded algebra $A^{\bullet}$ (DGA for short).
If the DGA $A^{\bullet}$
is graded commutative, then the bar complex $B(A^{\bullet},\epsilon)$
has a multiplication structure defined by the shuffle product.
Using this shuffle product, $H^0(B_{red}(A^{\bullet},\epsilon))$ has
a structure of Hopf algebra. 
Unfortunately Artin-Schreier DGA's
are not graded commutative in general. Therefore the shuffle product construction
does not work as it is. In Section \ref{Shuffle}, we define 
the shuffle product
on $\bold B_{red}(A,\epsilon)$ up to homotopy.

Here we give the definition of
Artin-Schreier DGA in the case of affine scheme $X=Spec(R)$, 
which is the easiest case. As a complex, $A^{\bullet}$
is a complex of length two, defined by the Artin-Schreier map
$$
A^0=R\overset{d}\to A^1=R:x\mapsto x^p-x.
$$
For an element $x\in R$, $(x)_i$ denote the element
in $A^i$ corresponding to $x$.
We introduce a multiplication structure $\cup$ on $A^{\bullet}$
as follows.
$$
((x)_0+ (\xi)_1)\cup((y)_0+ (\eta)_1)=
(xy)_0+ (x\eta+y^p\xi)_1.
$$
Actually, by this multipliation, $A^{\bullet}$ becomes
a DGA, by the direct calculation:
\begin{align*}
&d(((x)_0+(\xi)_1)\cup ((y)_0+(\eta)_1)) \\
=&
(x^py^p-xy)_1=(y^p(x^p-x))_1+(x(y^p-y))_1 \\
=&(x^p-x)_1\cup ((y)_0+(\eta)_1) +((x)_0+(\xi)_1)\cup (y^p-y)_1 \\
=&
d((x)_0+(\xi)_1)\cup ((y)_0+(\eta)_1)+(x)_0\cup d((y)_0+(\eta)_1) \\
&-(\xi)_1\cup d((y)_0+(\eta)_1))
\end{align*}

We give a Cech theoretic interpretation of the category 
$(NUR_{\bold F_p}/X)$ and the
Artin-Schreier DGA $A^{\bullet}$ in Section \ref{Recall}. 
To understand the relation between
Artin-Schreier DGA and the category $(NUR_{\bold F_p}/X)$, we will
recall the patching of algebras and pathching of DG categories.

In Section \ref{Shuffle}, we treat the product structure on
the cohomology of the bar complex. Though the Artin-Schreier DGA
$A^{\bullet}$ is not graded commutative, it is commutative up to homotopy.
In this section, we define a shuffle product up to higher homotopy.

{\bf Acknowledgement}.
The author would like to express his gratitude to Prof. L.Illusie and
H.Esnault for pointing out the relation about the work of 
\cite{K} and \cite{EK}. 
He would also express his thanks to H.Estanult for the discussions,
namely about the relation between pro-$p$ and $\bold F_p$ completions.
and L. Illusie for several comments on relation diagrams.
The author thanks R.Hain for letting him know about Thom-Whitney
construction for strict simplicial graded commutative DGA over
a characteristic zero filed. The author also thanks the anonymous referee
for remarks on the Conjecture \ref{homotopy commutative conjecture}.

\subsection{Convention}
\label{convention}
Let $\bold k$ be a field.
For a $\bold k$-linear abelian tensor category 
$\Cal C$, the category of complexes in $\Cal C$
is denoted as $K\Cal C$. By considering morphisms of complexes,
$K\Cal C$ becomes a $\bold k$-linear abelian category.

An object $A$ in $KK\Cal C$ is a complex $\{(A^{\bullet i},d)\}_i$ in
 $K\Cal C$. The differential $d$ of $A^{\bullet ,i}$ is called
the inner differential and the differential 
$\delta:A^{\bullet i} \to A^{\bullet i+1}$
is called the outer differential.
We define the associated simple complex $s(A)$ of $A$ by
$s(A)=\oplus_i A^{\bullet i}\otimes \bold k[-i]
=\oplus_i A^{\bullet i}[-i]$.
The total differential $D$ is defined as $d\otimes 1+\delta\otimes t$.
where $t:\bold k[-i]\to \bold k[-i-1]$ is a degree one map defined by
$\bold k[-i]^{i}\to \bold k[-i-1]^{i+1}:1\mapsto 1$.
The degree $-i$ element corresponding to $1$ in $\bold k[i]$
is denoted as $e^i$. The symbol $e^1$ is also denoted as $e$.
Using this notation, $s(A)$ can be written as $\oplus_i A^{\bullet i}e^{-i}$.

For objects
$A=A^{\bullet}$ and
$B=B^{\bullet}$ in $K\Cal C$, we define the tensor product $A\otimes B$
as an object in $K\Cal C$
by the rule $(A\otimes B)^p=\oplus_{i+j=p}A^i\otimes B^j$ and
$d(a\otimes b)=da\otimes b+(-1)^{deg(a)}a\otimes db$ for homogeneous
``elements'' $a$ and $b$ in $A$ and $B$.
We have a canonical switching isomorphsim 
\begin{equation}
\label{switching homomorphism}
\sigma:A\otimes B\simeq B\otimes A
\end{equation}
defeind by $\sigma(a\otimes b)=(-1)^{deg(a)deg(b)}b\otimes a$.
This switching rule can be extended to tensors of arbitrary number:
$$
\sigma:A_1\otimes\cdots \otimes A_n\to A_{\sigma(1)}\otimes
\cdots \otimes A_{\sigma(n)}
$$
for $\sigma\in \frak S_n$.
We always use an identification $A[p]\otimes B[q]\simeq A\otimes B[p+q]$
by the following composite:
\begin{align*}
A[p]\otimes B[q]&\simeq A\otimes \bold k[p]\otimes B\otimes \bold k[q] \\
&\overset{\sigma}\simeq A\otimes B\otimes \bold k[p]\otimes \bold k[q] \\
&\simeq A\otimes B\otimes \bold k[p+q].
\end{align*}
Here $\sigma$ is the switching homomorphism.

Let $A=A^{\bullet},B^{\bullet}$ be complexes in $\Cal C$
and $p$ be an integer.
A system of morphisms $A^i\to B^{i+p}$ indexed by $i\in \bold Z$
is called a homogeneous morphism of degree $p$ from $A$ to $B$.
Let $A_1,A_2,B_1$ and $B_2$ be complexes in the category $\Cal C$
and $\varphi:A_1\to A_2$ and $\psi:B_1\to B_2$ be homogeneous morphisms
of degree $p$ and $q$, 
respectively. Then the tensor product 
$$
\varphi\otimes \psi:A_1\otimes B_1 \to A_2\otimes B_2
$$
of $\varphi$ and $\psi$ is defined by the 
rule $(\varphi\otimes\psi)(a\otimes b)=(-1)^{q\deg(a)}
\varphi(a)\otimes \psi(b)$. 
Using this convention, the differential $d\otimes 1$ on $A[p]=A\otimes k[p]$
is $(d\otimes 1)(x\otimes e^p)=dx\otimes e^p$.
We remark that some references prefer sign convention 
$d_{A[p]}(x)=(-1)^{\deg(x)}d_A(x)$ for homogeneous
element $x \in A^{\bullet}[p]$, which is denoted $\bold k[p]\otimes
A^{\bullet}=e^pA^{\bullet}$ in this paper to aviod 
confusing convention ``$[p]A^{\bullet}$''.
The differential of a complex $A$ is a homogeneous map from $A$ to
$A$ of degree one.
Therefore the differential of the tensor complex $A\otimes B$ defined 
as above is written
as $d\otimes 1 +1\otimes d$.
The differential $d$ of $A$ can be considered as a homogeneous map of
degree zero $d\otimes t^{-1}$ from
$A$ to $A[1]$ by setting 
\begin{equation}
\label{differnetial as degree zero map}
x=x\otimes 1\mapsto (-1)^{\deg x}dx\otimes e.
\end{equation}
For homogeneous maps $\varphi_1:A_1\to
A_2,\varphi_2:A_2\to A_3$ and
$\psi_1:B_1\to
B_2,\psi_2:B_2\to B_3$, we have 
$$
(\varphi_2\otimes \psi_2)\circ(\varphi_1\otimes \psi_1)
=(-1)^{\deg\psi_2\cdot \deg\varphi_1}
(\varphi_2\circ\varphi_1)\otimes (\psi_2\circ\psi_1).
$$
A homogeneous map $L\otimes M\to N$ is extended to
$L[p]\otimes M[q] \to N[p+q]$ by the composite
\begin{equation}
\label{extention rule}
L\otimes \bold k[p]\otimes M\otimes \bold k[q]\simeq
L\otimes M\otimes  \bold k[p]\otimes \bold k[q]\simeq
L\otimes M\otimes  \bold k[p+q]\to N[p+q].
\end{equation}

For objects $A=A^{\bullet,\bullet}, B^{\bullet,\bullet}$  in $KK\Cal C$, 
the tensor
product $A\otimes B$ is defined as an object in $KK\Cal C$.
Then we have a natural isomorhpism in $K\Cal C$ 
$$
\nu:s(A)\otimes s(B)\simeq s(A\otimes B)
$$
defined by $\nu(a\otimes b)=(-1)^{ji'}a\otimes b$ for
$a\in A^{ij},b\in B^{i'j'}$. This isomorphism is compatible with
the switching of tensor and natural associativity isomorphism.

The set $\{m,m+1, \dots, n\}$ is denoted by $[m,n]$ and $[0,n]$ is
denoted by $[n]$. We define the category of simplexes $\Delta$ by setting the set
of objects to be the set of simplexes $\{[n]\mid n\geq 0\}$ and 
the set of morphisms to be the set of
weakly increasing maps between them.

We use copies $\overline{1}, \dots, \overline{n}$
of numbers $1,\dots, n$ which is totally ordered in the natural way.

\section{$(UR_{\bold F_p}/X)$ and quotient by relation}
\label{Recall}

In this subsection, we give an interpretation of objects 
in $(UR_{\bold F_p}/X)$ 
as locally free sheaves
on the quotient ``space'' by a relation diagram.
\begin{definition}
\begin{enumerate}
\item
The following diagram $R$ of schemes
\begin{equation}
\label{relation diagram}
R:X
\begin{matrix}
\overset{f_0}\to \\ \underset{f_1}\to
\end{matrix}
Y
\end{equation}
is called a relation diagram.
\item
For a relation diagram $R$, we define 
a (locally free) $\Cal O_{R}$-module of finite rank as a pair
$(\Cal M,\varphi)$ consisting of (locally free) $\Cal O_Y$ 
module $\Cal M$ of finite rank and an isomorphism
$f^*_0\Cal M\simeq f^*_1\Cal M$,
Since we have
$can:f^*_0\Cal O_Y\simeq\Cal O_X\simeq f^*_1\Cal O_Y$,
$(\Cal O_Y,can)$ is an $\Cal O_R$ module. It is written
$\Cal O_R$ for simplicity.
\item
Let $R$ be a relation diagram (\ref{relation diagram}) and 
$\Cal M=(\Cal M, \varphi)$ be a locally free $\Cal O_R$-module
of finite rank.
A filtration $F^{\bullet}_{\Cal M}$
on $\Cal M$
is called a nilpotent filtration if
\begin{enumerate}
\item
the filtrations $f^*_1 F^{\bullet}_{\Cal M}$ and
$f^*_2 F^{\bullet}_{\Cal M}$ coincide under the isomorphism
$\varphi$, and
\item
the associated graded $\Cal O_{R}$ modules are sum of copies of
$\Cal O_R$. 
\end{enumerate}
A locally free $\Cal O_R$-module of finite rank with a nilpotent filtration
is called a nilpotent $\Cal O_R$-module.
\end{enumerate}
\end{definition}
\begin{example}
Let
$$
R:Spec(\bold k)
\begin{matrix}
\overset{f_0}\to \\ \underset{f_1}\to
\end{matrix}
Spec(\bold k[t])
$$
be a relation diagram, where $f_0$ and $f_1$
are the evaluation map at $0$ and $1$, respectively.
Then the category of
$\Cal O_{R}$-modules
is naturally equivalent to the category of $\Cal O_C$-modules 
where $C=Spec(Ker(\bold k[t]\overset{ev_0-ev_1}\to \bold k)))$
is the nodal affine line.
\end{example}
In the same way, a diagram 
\begin{equation}
\label{DGA relation}
A^{\bullet,0}
\begin{matrix}
\overset{f_0}\to \\ \underset{f_1}\to
\end{matrix}
A^{\bullet,1}
\end{equation}
of DGA's is called a relation diagram of DGA's.
For a relation diagram (\ref{DGA relation}), we denote $s(A)$ 
by the 
associated simple complex of
$$
\underset{\text{degree $0$}}{A^{\bullet,0}}\overset{f_1-f_0}\longrightarrow
\underset{\text{degree $1$}}{A^{\bullet,1}}.
$$
Then $s(A)=A^{\bullet 0}\oplus 
A^{\bullet 1}e^{-1}$ is equipped with a structure $\cup$ of DGA by
setting 
\begin{align*}
a\cup b&=a_0b_0+a_1b_1+f_0(a_0)e^{-1}b_{01}+a_{01}f_1(b_1)e^{-1} \\
&=a_0b_0+a_1b_1+(-1)^{\deg b_{01}}f_0(a_0)b_{01}e^{-1}+a_{01}f_1(b_1)e^{-1}
\end{align*}
where $a=a_0+a_1+a_{01}e^{-1},b=b_0+b_1+b_{01}e^{-1}$, where
$a_{0},b_{0}\in A_{0}$, $a_{1},b_{1}\in A_{1}$ and
$a_{01},b_{01}\in A_{01}$ are homogeneous elements. 
Note that 
$\deg(b_{01}e^{-1})=\deg(b_{01})+1$.

Since the unipotencies for two categories $(LC_{\bold F_p}/X_{et})$ and
$(UR_{\bold F_p}/X)$ correspond to each other by the functor stated
just after the statement of Theorem \ref{Katz theorem}, the following proposition
is a consequence of this theorem.
\begin{proposition}
\label{restatement of Katz}
Let $X$ be a scheme of finite type over $\bold F_p$ and $R$ be
the relation diagram
\begin{equation}
\label{frb-id rel daig sch}
X
\begin{matrix}
\overset{id_X}\to \\
\underset{F_X}\to
\end{matrix}
X.
\end{equation}
Then the category of locally free sheaves of finite rank on $\Cal O_R$ is
equivalent to the category of $\bold F_p$- smooth etale
sheaves of finite rank on $X$.
Moreover the category of nilpotent $\Cal O_R$-modules is
equivalent to the category of nilpotent $\bold F_p$ smooth
sheaves on $X$.
\end{proposition}

\section{Bar complex and Cech complex}
\subsection{The bar complex and nilpotent connections}
We recall the definition of the reduced bar complex 
and the simplicial bar complex
of a DGA $A^{\bullet}$
over a field $\bold k$ with
an augmentation $\epsilon:A^{\bullet}\to \bold k$. 
The multiplication of $A^{\bullet}$ is written as
``$\cdot$''.
Let
$I=Ker(\epsilon:A^{\bullet}\to \bold k)$ be the 
augmentation ideal of $\epsilon$.

We define the double complex
$$
\begin{matrix}
\bold B_{red}(A^{\bullet},\epsilon):&
\cdots \overset{d_B}\to &
I^{\otimes 2}&\overset{d_B}\to& 
I^{\otimes 1} &\to& \bold k,\\  
& & \Vert
& & \Vert
& & \Vert \\
& & \bold B_{red}(A^{\bullet},\epsilon)^{-2}
& & \bold B_{red}(A^{\bullet},\epsilon)^{-1}
& &  \bold B_{red}(A^{\bullet},\epsilon)^{0}
\end{matrix}
$$
where $d_B:I^{\otimes p} \to I^{\otimes (p-1)}$ is given
by
\begin{align}
\label{definition of reduced bar differential}
d_B(x_1 \otimes x_2 \otimes \cdots \otimes x_p) = &
(-1)^p\epsilon(x_1) \otimes x_2\otimes \cdots \otimes x_p \\
\nonumber
&+
 \sum_{i=1}^{p-1}
(-1)^{p-i}x_1 \otimes \cdots \cdots \otimes (x_{i}\cdot x_{i+1})\otimes \cdots
\otimes x_p \\
\nonumber
&+
x_1 \otimes x_2\otimes \cdots \otimes x_{p-1}\otimes \epsilon(x_p).
\end{align}
The associated simple complex 
$B_{red}(A^{\bullet},\epsilon)=
\oplus_i\bold B_{red}(A^{\bullet},\epsilon)^{-i}[i]$
is called Chen's reduced bar complex of $(A^{\bullet},\epsilon)$.

We define the simplicial bar complex of $A^{\bullet}$.
For a sequence of integers $\alpha=(\alpha_0<\dots <\alpha_n)$, we set
$\bold B_{simp,\alpha}(A^{\bullet},\epsilon)={A^\bullet}^{\otimes n}$.
The length of $\alpha$ is by definition $n$.
We write this tensor product as 
$$
\bold k\overset{\alpha_0}\otimes A\overset{\alpha_1}\otimes
\cdots \overset{\alpha_{n-1}}\otimes A \overset{\alpha_n}\otimes \bold k
$$
to distinguish the index $\alpha$.

For indices $\alpha=(\alpha_0, \dots, \alpha_n)$ and 
$\beta=(\alpha_0,\cdots,\widehat{\alpha_i},\cdots,\alpha_n)$,
we define a homomorphism of complexes
$d_{bar}:\bold B_{simp,\alpha}(A,\epsilon) \to \bold B_{simp,\beta}(A,\epsilon)$
by
$$
d_{bar}:1\overset{\alpha_0}\otimes x_1\overset{\alpha_1}\otimes \cdots 
\overset{\alpha_{n-1}}\otimes
x_n\overset{\alpha_n}\otimes 1 \mapsto
\begin{cases}
(-1)^n\epsilon(x_1)\overset{\alpha_1}\otimes x_2
\cdots 
x_n\overset{\alpha_n}\otimes 1 & (i=0) \\
(-1)^{n-i} 1\overset{\alpha_0}\otimes 
\cdots 
\overset{\alpha_{i-1}}\otimes x_i\cdot x_{i+1}
\overset{\alpha_{i+1}}\otimes \cdots
\overset{\alpha_n}\otimes 1 & (0<i<n) \\
 1\overset{\alpha_0}\otimes x_1 
\cdots 
x_{n-1}
\overset{\alpha_{n-1}}\otimes \epsilon(x_n) & (i=n) \\
\end{cases}
$$
We set
$\bold B^{-n}_{simp}(A,\epsilon)= 
\bigoplus_{\mid \alpha\mid =n}\bold B_{simp,\alpha}(A,\epsilon)$.
Using $d_{bar}$, we have the following double complex:
$$
\cdots \overset{d_{bar}}\to \bold B^{-2}_{simp}(A, \epsilon) 
\overset{d_{bar}}\to \bold B^{-1}_{simp}(A, \epsilon) 
\overset{d_{bar}}\to \bold B^0_{simp}(A, \epsilon)\to 0
$$
The associated simple complex $B_{simp}(A^{\bullet},\epsilon)$
is called the simplicial bar complex of $(A^{\bullet},\epsilon)$.
By Theorem 5.2 of \cite{T}, we have the following proposition.  
\begin{proposition}
Let $\alpha=(\alpha_0< \dots< \alpha_n)$ be a sequence of integers.
The map 
$$
B_{simp,\alpha}\to I^{\otimes n}:
1\otimes x_1\otimes \cdots \otimes x_n\otimes 1 \mapsto
\pi(x_1) \otimes \cdots \otimes \pi(x_n)
$$
defines a homomorphism of double complexes
$\bold B_{simp}(A, \epsilon)\to \bold B_{red}(A, \epsilon)$,
which induces a quasi-isomorphism 
$B_{simp}(A, \epsilon)\to B_{red}(A, \epsilon)$.
\end{proposition}
We introduce a differential graded coalgebra structure on the double complex
$\bold B_{red}(A^{\bullet},\epsilon)$: 
$$
\mu:\bold B_{red}(A^{\bullet},\epsilon) \to
\bold B_{red}(A^{\bullet},\epsilon)\otimes
\bold B_{red}(A^{\bullet},\epsilon)
$$
defined by 
$$
\mu(x_1\otimes \cdots \otimes x_n)=\sum_{i=0}^n
(x_1\otimes \cdots \otimes x_i)\otimes 
(x_{i+1}\otimes \cdots \otimes x_n).
$$

The homomorphism $\mu$ defines a comultiplication  
on the associated simple complex
$B_{red}(A^{\bullet},\epsilon)$ and its cohomology
$H^0(B_{red}(A^{\bullet},\epsilon))$:
\begin{equation}
\label{comult on h0}
\mu:H^0(B_{red}(A^{\bullet},\epsilon))\to
H^0(B_{red}(A^{\bullet},\epsilon))\otimes 
H^0(B_{red}(A^{\bullet},\epsilon))
\end{equation}
Via this comultiplication, $H^0(B_{red}(A^{\bullet},\epsilon))$
becomes a coalgebra over $\bold k$.

We define nilpotent $A^{\bullet}$-connections for a connected
DGA.
\begin{definition}
\label{def of nilp conn}
Let $A^{\bullet}$ be a DGA over a field $\bold k$.
\begin{enumerate}
\item
$A^{\bullet}$ is said to be connected if $H^i(A^{\bullet})=0$ for $i<0$
and $H^{0}(A^{\bullet})=\bold k$.
(This condition is sometimes called the cohomological connectedness.)
\item
\label{connction finite dimensional}
Let $M$ be a finite dimensional $\bold k$ vector space. 
A map $\nabla:M \to A^1\otimes M$
is called a nilpotent $A^{\bullet}$-connection 
if there exists a finite decreasing filtration $\{F^iM\}_i$
such that $\nabla (F^iM)\subset A^1\otimes F^{i+1}$
and $F^pM=0$, $F^qM=M$ for some $p,q\in \bold Z$.
A nilpotent $A^{\bullet}$-connection is said to be integrable if
the following diagram is commutative:
$$
\begin{matrix}
M &\overset{\nabla}\to &A^1\otimes M \\
\nabla\downarrow & & \downarrow d\otimes 1_M \\
A^1\otimes M &\overset{(\mu\otimes 1)\circ (1_A\otimes \nabla)}\to
& A^2\otimes M.
\end{matrix}
$$
A pair $(M,\nabla)$
of a vector space $M$ and an integrable nilpotent 
$A^{\bullet}$-connection $\nabla$ on $M$ 
is called an integrable nilpotent $A^{\bullet}$-connection for short.
\item
Let $\Cal M_1=(M_1,\nabla_1)$ and $\Cal M_2=(M_2,\nabla_2)$ be 
integrable nilpotent connections. A $\bold k$ linear map
$\varphi:M_1\to A^0\otimes M_2$ is called a homomorphism of 
integrable nilpotent connections
from $\Cal M_1$ to $\Cal M_2$ if the following diagram commutes:
$$
\begin{matrix}
M_1 & \overset{\nabla_1}\to & A^1\otimes M_1 &
\overset{1_A\otimes \varphi}\to & A^1\otimes A^0 \otimes M_2 \\
\varphi\downarrow & & & &\downarrow\cup\otimes 1_{M_2} \\
A^0\otimes M_2 & &
\overset{(\cup\otimes 1_{M_2})(1_A\otimes \nabla_2)+d\otimes 1_{M_2}} 
\longrightarrow
  & & A^1\otimes M_2. \\
\end{matrix}
$$
\item
Let $\Cal M_1=(M_1,\nabla_1)$ and $\Cal M_2=(M_2,\nabla_2)$ be 
two integral nilpotent $A^{\bullet}$-connections.
Let $\varphi_0$ and $\varphi_1$ be two homomorphisms from
$\Cal M_1$ to $\Cal M_2$. 
The homomorphisms $\varphi_0$ and $\varphi_1$ are homotopy equivalent
if 
there exists a $\bold k$ linear map $h:M_1 \to A^{-1}\otimes M_2$ such that
the difference $\varphi_1-\varphi_0$ is equal to the
sum of two maps
\begin{align*}
& M_1\overset{h}\to A^{-1}\otimes M_2
\overset{(\cup \otimes 1_{M_2})(1\otimes \nabla_2)+d\otimes 1_{M_2}}\longrightarrow
A^0\otimes M_2\\
& M_1\overset{\nabla_1}\to A^1\otimes M_1 \overset{1\otimes h}
\to A^1\otimes A^{-1}\otimes M_2 \to A^0\otimes M_2.
\end{align*}
A $\bold k$ linear map $h$ satisfying the above condition is called a 
homotopy of $\varphi_0$ and $\varphi_1$.
\end{enumerate}
\end{definition}
The category of integrable nilpotent connections is denoted as
$(INC_A)$. It is easy to see that the set of null homotopic 
morphisms in $(INC)$ forms an ideal in the set of morphisms
and we get a quotient category by replacing the morphism
by their homotopy equivalence classes. This quotient category
is denoted by $(HINC_A)$.
We refer the following theorem in \cite{T}. 
\begin{theorem}[Theorem 7.6 in \cite{T}]
Let $A^{\bullet}$ be a connected DGA.
The category $(HINC_A)$ is equivalent to the category of 
$H^0(B(A^{\bullet},\epsilon))$-comodules.
\end{theorem}

\subsection{Cech complexes}
We recall Cech complexes for a coverings and surjecitve
morphisms.

\subsubsection{}
Let $I$ be a finite ordered set and $\Cal P(I)$ (resp. $\Cal P^+(I)$) be 
the set of non-empty subsets (resp. subsets) of $I$. 
Then $\Cal P(I)$ and $\Cal P^+(I)$ become categories by
inclusions. For a category $\Cal C$, a covariant functor
from $\Cal P(I)$ (resp. $\Cal P^+(I)$ to $\Cal C$ is called a 
(resp. an augmented strict cosimplicial object)
strict cosimplicial object
of $\Cal C$ indexed by $I$.

For an element $\bold i=(i_0<\dots< i_p)\in \Cal P(I)$,
we define $\mid \bold i\mid =p$.
We set
$D_k(\bold i)=(i_0,\dots, \widehat{i_k},\dots, i_p)$ for
$\bold i=(i_0, \dots, i_p)$.
We denote
$\bold i_1*\bold i_2\vdash\bold i$ if
$\bold i_1=(i_0, \dots, i_k), \bold i_2=(i_k, \dots, i_p)$ and
$\bold i=(i_0,\dots, i_k,\dots, i_p)$.
Let $\Cal C$ be an abelian category and 
$\{A_J\}_{J\in \Cal P(I)}$ a strict cosimplicial object of $\Cal C$.
Then we define the Cech complex $\check{\bold C}(I,A)$ by
$$
0\to \prod_{\mid J\mid=0}A_J \to \prod_{\mid J\mid=1}A_J \to \cdots,
$$
where the differential $d$ is given by the formula
$$
(da)_{\bold i}=\sum_{k=0}^p(-1)^k a_{D_k(\bold i)}.
$$
\subsubsection{}
Let $I_0, \dots, I_m$ be finite non-empty ordered sets. 
The sequence $(I_0, \dots, I_m)$ is denoted as $\bold I$.
The covariant
functor $A$ from $\Cal P(I_1)\times \cdots \times\Cal P(I_m)$
to a category $\Cal C$ is called a strict polycosimplicial object
in $\Cal C$ indexed by $\bold I$. 
For a strict polycosimplicial object in an abelian category
$\Cal C$, we define Cech complex 
$\check{\bold C}(\bold I,A)\in K\Cal C$ of $A$ by
$$
0\to \prod_{\mid \bold i_0\mid +\cdots+ \mid \bold i_m\mid=0}
A_{\bold i_0, \dots, \bold i_m} \to
\prod_{\mid \bold i_0\mid +\cdots+ \mid \bold i_m\mid=1}
A_{\bold i_0, \dots, \bold i_m} \to
$$
Here the differential $d$ is defined by
$$
(da)_{\bold i_1,\dots, \bold i_m}=
\sum_{0\leq l\leq m, 0<\mid \bold i_l\mid}
\sum_{k=0}^{\mid i_l\mid}(-1)^
{\mid i_0\mid+\cdots \mid i_{l-1}\mid +l+k} 
a_{\bold i_0, \dots, D_k(\bold i_l),\dots, \bold i_m}.
$$

\subsubsection{}
Let $\Delta$ be the category of simplexes and $\Cal C$ be 
an abelian category.
A covariant functor from $\Delta$ to $\Cal C$ is called a
cosimplicial object in $\Cal C$. If $\Cal C$ is an abelian category,
then for a cosimplicial object $\{A^i\}_{i\geq 0}$, we define
the associated total complex $\bold t(A)$ as an object of
$K\Cal C$ by
$$
0\to A^0 \to A^1 \to \cdots.
$$
\begin{definition}[Alexander Whitney products]
\label{A-W prodct}
\begin{enumerate}
\item
\label{A-W prod strict cosimplical}
Let $A$ be a strict cosimplicial DGA indexed by $I$. 
We define the pruduct $\mu$ 
$$
\mu:\check{\bold C}(I,A)\otimes \check{\bold C}(I,A)\to 
\check{\bold C}(I,A)
$$
by 
\begin{equation}
\label{Alexander-Whitney product primitive}
\mu(a\otimes b)_{\bold i}=
\sum_{\bold i_1*\bold i_2\vdash\bold i}
 a_{\bold i_1}\cdot b_{\bold i_2}
\end{equation}
The associated simple complex $s(\check{\bold C}(I,A))$ of the Cech complex 
of $A$ is denoted as $\check{C}(I,A)$.
Then $\check{C}(I,A)$ becomes an associative DGA by the multiplication $\mu$. 
\item
\label{polysimp}
Let $A$ be a strict polycosimplicial DGA indexed by 
$\bold I=(I_0, \dots,I_m)$. 
We define the product $\mu$ 
$$
\check{\bold C}(\bold I,A)\otimes  \check{\bold C}(\bold I,A) 
\to \check{\bold C}(\bold I,A)
$$
by
$$
\mu(a\otimes b)_{\bold i_0, \dots, \bold i_m}=
\sum_{\bold j_0*\bold k_0\vdash \bold i_0,\dots, 
      \bold j_m*\bold k_m\vdash \bold i_m}
a_{\bold j_0,\dots, \bold j_m}\cdot
a_{\bold k_0,\dots, \bold k_m}
$$
the associated simple complex $\check{C}(\bold I,A)$ 
becomes an associative DGA. 
\item
\label{A-W prod simplical}
Let $\{A^{i}\}_i$ be a cosimplicial DGA. We define the product
$$
\mu:A^l\otimes A^m \to A^{l+m},
$$
by 
$$
\mu(a\otimes b)=i_{l+m,l,*}(a)\cdot j_{l+m,m,*}(b),
$$
where $i_{l+m,l,*}$ and $j_{l+m,m,*}$ are morphism of DGA's
induced by 
\begin{align*}
&i_{l+m,l}:[l]\to [l+m]:s\mapsto s \text{ for } s\in [l] \\
&j_{l+m,m}:[m]\to [l+m]:s\mapsto s+l \text{ for } s\in [m].
\end{align*}
By this product, 
the associated simple complex of $\bold t(A^{\bullet})$ 
becomes an associative DGA. It is denoted by $t(A^{\bullet})$.
\end{enumerate}
The products defined as above are called Alexander-Whitney products.
\end{definition}

\subsubsection{}
We set $I=I_0\coprod \cdots \coprod I_m$ and introduce a total order
on $I$. 
For simplicity, we set $I_0=[0,i_1-1],I_1=[i_1,i_2-1],\cdots
,I_m=[i_m,k]$ and $I=[0,k]$.
For a non-empty subset 
$\bold j=(j_0, \dots,j_l)\in \Cal P(J)$ of $J=[m]$,
we set $I_{\bold j}=I_{j_0}\coprod\cdots \coprod I_{j_l}$
and $\bold I_{\bold j}=(I_{j_0},\cdots,I_{j_l})$.
Then
we have an inclusion $\Cal P(I_{j_0})\times \cdots \times \Cal P(I_{j_l})
\to \Cal P(I)$ obtained by taking the union of subsets of
$I_{j_0},\dots, I_{j_l}$.
Let $A$ be a strict cosimplicial object indexed by $I$ in $\Cal C$. 
By restricting the functor $A$ to 
$\Cal P(I_{j_0})\times \cdots \times \Cal P(I_{j_{l}})$, we have a strict 
polycosimplicial object $A_{\bold j}$ indexed by $\bold I_{\bold j}$.
Thus we have a Cech complex 
$\check{\bold C}(\bold I_{\bold j},A_{\bold j})\in K\Cal C$.
Let $\bold j=(j_0,\dots, j_l)\in \Cal P(J)$ and $0\leq q\leq l$ and set
$D_q(\bold j)=(j_0,\dots, \widehat{j_q},\dots, j_l)$.
We define a 
homomorphism 
$\sigma_{\bold j,q}:\check{\bold C}(\bold I_{D_q(\bold j)},A_{D_q(\bold j)})\to 
\check{\bold C}(\bold I_{\bold j},A_{\bold j})$
in $K\Cal C$ by
$$
\sigma(a)_{\bold i_{j_0},\dots, \bold i_{j_l}}=\begin{cases}
(-1)^{p-i_{j_q}} 
a_{\bold i_{j_0},\dots,\widehat{\bold i_{j_q}} ,\dots, \bold i_{j_l}}
&\text{ if }\bold i_{j_q}=\{p\}, p\in I_{j_q} \\
0 &\text{ otherwise }.
\end{cases}
$$
As a consequence, the system 
$\{\check{\bold C}(\bold I_\bold j,A_{\bold j})\}_{\bold j}$ 
forms a strict cosimplicial object in $K\Cal C$ indexed
by $J=[m]$. It is denoted as $\check{\bold C}(I/J, A)$ and called the 
relative Cech complex for $I/J$. 
Thus we have the Cech complex 
$\check{\bold C}(J,\check{\bold C}(I/J,A))$ in $KK\Cal C$.
The assocated simple complex in $K\Cal C$ is denoted as
$\check{C}(J,\check{\bold C}(I/J,A))$.

\begin{proposition}
The complex
$\check{C}(J,\check{\bold C}(I/J,A))$ in $K\Cal C$ is isomorphic to
$\check{\bold C}(I,A)$.
\end{proposition}

Let $A$ be a strict cosimplicial DGA indexed by $I$.
For an element $\bold j \in \Cal P(J)$,
$\check{C}(\bold I_\bold j,A_\bold j)$ 
is a DGA under the multiplication defined in
Definition \ref{A-W prodct}(\ref{polysimp})
and this correspondece gives rise
to a strict cosimplical DGA indexed by $J$ which is denoted as
$\check{C}(I/J,A)$.

\begin{proposition}
The DGA $\check{C}(J,\check{C}(I/J,A))$ is isomorphic
to $\check{C}(I,A)$.
\end{proposition}
Let $I$ and $J$ be finite ordered sets and $\varphi:I\to J$ 
be a non-decreasing map. Let $\Cal C$ be an abelian category and
$A$ and $B$ be strict cosimplicial objects in $\Cal C$ indexed by
$I$ and $J$, respectively.
A $\varphi$ morphism from $B$ to $A$ is a set of morphisms
$\tau_{\bold i,\bold j}:B_{\bold j}\to A_{\bold i}$ for 
$\bold i\in \Cal P(I), \bold j\in \Cal P(J),\varphi(\bold i)\subset
\bold j$ 
such that the following diagram commutes
$$
\begin{matrix}
B_{\bold j}&\to & A_{\bold i} \\
\downarrow & & \downarrow\\
B_{\bold j'}&\to & A_{\bold i'} 
\end{matrix}
$$
for all $\bold i\subset\bold i'\in \Cal P(I),
\bold j\subset\bold j'\in \Cal P(J)$
such that $\varphi(\bold i)\subset \bold j,\varphi(\bold i')\subset \bold j'$.
For a $\varphi$ morphism from $B$ to $A$, we define a morphism of complexes
$$
\varphi_*:\check{C}^p(B)=\prod_{\mid\bold j\mid=p}B_\bold j\to 
\check{C}^p(A)=\prod_{\mid \bold i\mid=p}A_\bold i
$$ 
by
$$
\varphi_*(a)_{i_0,\dots,i_p}=
\begin{cases}
\tau(a_{\varphi(i_0),\dots,\varphi(i_p)}) &\text{ if }
\varphi(i_0)<\dots<\varphi(i_p)
\\
0&\text{ otherwise}.
\end{cases}
$$

\section{Bar complex of Artin-Schreier DGA for $\bold F_p$-scheme}
\label{AS DGA}

\subsection{Cech complex of Artin-Schreier DGA}
Let $X=Spec(A)$ be an irreducible affine $\bold F_p$-scheme of finite
type with an $\bold F_p$-valued
point $x$ of $X$. 
By the diagram (\ref{frb-id rel daig sch})
in Proposition \ref{restatement of Katz}, we have
a diagram (\ref{DGA relation}), 
since the commutative algebra $R$ is a DGA over $\bold F_p$.
The total complex 
$AS(A)=(A\overset{F_A-id_A}\to A)$ 
has a strucuture of DGA by the last papragraph
and is called the Artin-Schreier DGA of $A$.
The $\bold F_p$-valued point $x$ defines a augmentation $\epsilon$
of $AS(A)$.
For a scheme $U$, by attatching
$AS(\Gamma(U,\Cal O_U))$, we get a presheaf of complex
on $(Sch/\bold F_p)$ which is denoted as ${\Cal AS}$. 

Let $\Cal W=\{W_i\to X\}_{i \in I}$ be a family of $X$-schemes indexed by a
totally ordered set $I$.  
We set $W_{J}=W_{j_0}\times_X\cdots \times_X W_{j_n}$ for a subset
$J=(j_0,\dots,j_n)$ of $I$. 
Then by attatching $AS(W_J)$ to an element $J\in \Cal P(I)$,
we get a strict cosimplicial DGA indexed by $I$, 
which is denoted as $AS(\Cal W)$.
Then the Cech complex $\check{C}(I,AS(\Cal W))$ defined in
Definition \ref{A-W prodct}
(\ref{A-W prod strict cosimplical})
is a DGA.

Let $U\to X$ be a morphism of $\bold F_p$-scheme. Let $U_n$ be the 
$n+1$ times fiber product $U\times_X\cdots \times_X U$ of $U$
over $X$. For a presheaf $\Cal A$ of DGA on $X_{et}$,
we have a cosimplicial DGA $\{\Cal A(U_n)\}_n$, which
is denoted as $\Cal A(U/X)$. By 
Definition \ref{A-W prodct}
(\ref{A-W prod simplical}), we have a DGA
$t(\Cal A(U/X))$. 
By considerng the constant sheaf $\bold F_p$ on $X$,
the associate simple complex $t(\bold F_p(U/X))$ of
$$
\bold F_p(U)\to \bold F_p(U\times_X U) \to\bold F_p(U\times_X U\times_X U) \to \cdots.
$$
becomes a DGA. Similarly, we have
a DGA $t(\Cal AS(U/X))$ arising from the sheaf ${\Cal AS}$ of Artin-Schreier DGA.

\subsection{Bar complex of Artin-Schreier DGA}

Let $X$ be a
separated irreducible scheme of finite type over $\bold F_p$
with an $\bold F_p$-valued point $x$ of $X$.
Let $\Cal W=\{W_i \to X\}_{i\in I}$ be a finite affine covering of $X$
indexed by $I$.
We choose $i\in I$ such that the base point
$x$ factors through $Spec(\bold F_p) \to W_i\subset X$. By using
this $x$, we get an augmentation 
$\epsilon:\check{C}(I,AS(\Cal W))\to \bold F_p$.

\begin{theorem}
Let $X$ be a separted irreducible scheme of finite type over $\bold F_p$
and $x$ be an $\bold F_p$-valued point of $X$.
The category of $H^0(B_{red}(\check{C}(I,AS(\Cal W)),\epsilon))$-comodules 
is equivalent
to the category of nilpotent smooth etale $\bold F_p$-sheaves of finite
rank on $X$. 
\end{theorem}

\begin{corollary}
Let $A$ be an $\bold F_p$-algebra of finite type.
Assume that $Spec(A)$ is irreducible.
The category of $H^0(B_{red}(AS(A),\epsilon))$-comodules is equivalent
to the category of smooth etale $\bold F_p$-sheaves on $Spec(A)$. 
\end{corollary}

\begin{proof}
Let $U_i\to W_i$ be an etale covering of $W_i$.
Then for a subset $J$ of $I$, $U_J=U_{j_0}\times_X\cdots \times_X U_{j_n}$
is an etale covering of $W_J$. The union $U=\coprod_{i\in I}U_i$
is an etale covering of $X$.
For a commutative diagram 
$$
\begin{matrix}
S' &\to& S \\
\downarrow & & \downarrow \\
T' & \to&  T
\end{matrix}
$$
where $S'\to T'$ and $S\to T$ are etale coverings,
we have a homomorphism of DGA
$$
t(\Cal A(S/T)) \to t(\Cal A(S'/T')).
$$
Therefore by attaching $t(\Cal A(U_J/W_J))$
to $J$, we obtain a strict cosimplicial DGA indexed by $I$, which is denoted as
$t(\Cal A(\Cal U/\Cal W))$.
Therefore we have a homomorphism of complex
$$
t(\Cal A(U/X))\to \check{C}(I,t(\Cal A(\Cal U/\Cal W))).
$$
By taking a inductive limit on $\Cal U$ for the diagrams

\setlength{\unitlength}{0.65mm}
\begin{picture}(200,25)(-14,0)
\put(50,20){\vector(1,0){25}}
\put(47,16){\vector(1,-1){10}}
\put(79,16){\vector(-1,-1){10}}

\put(60,0){$W_i$,}
\put(40,18){$U_i'$}
\put(80,18){$U_i$}

\end{picture}
\noindent
we have the following quasi-isomorphism for an etale sheaf $\Cal A$ on 
$X_{et}$:
\begin{equation}
\label{quasi-iso see milne}
\lim_{\underset{\Cal U}\to}
t(\Cal A(U/X))\to 
\lim_{\underset{\Cal U}\to}
\check{C}(I,t(\Cal A(\Cal U/\Cal W))).
\end{equation}
Thus we have the following commutative diagram of quasi-isomorphisms:
$$
\begin{matrix}
\displaystyle\lim_{\underset{\Cal U}\to}
t(\bold F_p(U/X))&\to& 
\displaystyle\lim_{\underset{\Cal U}\to}
\check{C}(I,t(\bold F_p(\Cal U/\Cal W))) \\
\downarrow & &\downarrow \\
\displaystyle\lim_{\underset{\Cal U}\to}
t(\Cal AS(U/X))&\to &
\displaystyle\lim_{\underset{\Cal U}\to}
\check{C}(I,t(\Cal AS(\Cal U/\Cal W))) 
\end{matrix}
$$
The horizontal quasi-isomorphisms is those from the quasi-isomorphism
 (\ref{quasi-iso see milne})
and the vertical quasi-isomorphism on the left comes from the
 Artin-Schreier exact sequence for etale topology.
$$
0\to \bold F_p \to \Cal O_X \overset{F-id}\to \Cal O_X \to 0.
$$
By the descent theory for coherent sheaves for etale topology, 
we have an exact sequence
$$
0\to \Gamma(W_J,\Cal O_{W_J})\to \Gamma(U_J,\Cal O_{U_J})
\to \Gamma(U_J\times_{W_J}U_J,\Cal O_{U_J\times_{W_J}U_J})\to \cdots
$$
and as a consequence, the morphism
$AS(W_J)\to\Cal AS(U_J/W_J)$ is a quasi-isomorphism.
Therefore we have a quasi-isomorphism of DGA's:
$$
\check{C}(I,AS(\Cal W))) 
\to \displaystyle\lim_{\underset{\Cal U}\to}
\check{C}(I,t(\Cal AS(\Cal U/\Cal W))).
$$
Let $\bar x$ be a geometric point over $x$ and $\bar \epsilon$
be a lift $\bar x\to \underset{\ot}\lim U_i$ of $x$.
Then we have a homomorphism of algebras
$\underset{\to}\lim \bold F_{p}(U_i)\to \bold F_p$
and it defines an augmentation of $t(\bold F_p(U/X))$, which
is also denoted as $\bar\epsilon$.
Using the following diagram,
\begin{equation}
\label{diagram 1}
\end{equation}
\setlength{\unitlength}{0.65mm}
\begin{picture}(200,38)(0,-3)
\put(45,20){\vector(1,0){15}}
\put(110,20){\vector(-1,0){15}}
\put(150,20){\vector(-1,0){15}}
\put(77,31){\vector(0,-1){7}}
\put(123,8){\vector(0,1){7}}
%
\put(33,16){\line(0,-1){12}}
\put(33,4){\vector(1,0){85}}
\put(35,25){\vector(2,1){20}}
\put(155,16){\vector(-2,-1){25}}
\put(160,25){\line(0,1){12}}
\put(160,37){\vector(-1,0){62}}

\put(121,2){$\bold F_p$}
\put(10,18){$\underset{\to}\lim \bold F_p(U_i/W_i)$}
\put(65,18){${\Cal AS}(\overline{\bold F_p}/\overline{\bold F_p})$}
\put(114,18){${\Cal AS}(\overline{\bold F_p})$}
\put(154,18){${\Cal AS}(W_i)$}
\put(60,35){$\underset{\to}\lim {\Cal AS}(U_i/W_i)$}
\put(80,5){$\bar\epsilon$}
\put(145,8){$\epsilon$}
\put(143,21){$\tilde\epsilon$}
\put(33,33){{\tiny quasi-iso}}
\put(120,34){{\tiny quasi-iso}}
\put(95,22){{\tiny quasi-iso}}
\put(107,10){{\tiny quasi-iso}}

\end{picture}
\noindent
we have an isomorphism of coalgebras:
$$
\underset{\underset{\Cal U}\to}\lim\ 
H^0(B(t(\bold F_p(U/X)),\bar \epsilon))\simeq
H^0(B(\check{C}(I,AS(\Cal W)),\epsilon)).
$$
By the following proposition,
we have the theorem.
\end{proof}
\begin{proposition}
The homotopy category of integrable unipotnent 
$\displaystyle\lim_{\underset{\Cal U}\to}t(\bold F_p(U/X))$-
connections is equivalent to the category of nilpotent smooth etale
$\bold F_p$ sheaves.
\end{proposition} 
\begin{proof}
For a scheme $X$, the constant sheaf on $X$ with values in $M$ is denoted
by $M_X$.
Let $(M,\nabla)$ be a 
$\displaystyle\lim_{\underset{\Cal U}\to}t(\bold F_p(U/X))$-connection,
and $F$ be a nilpotent filtration of $M$ for the connection.
We construct an etale sheaf on $X$ by descending a constant
sheaf on an etale covering of $X$.

Since $M$ is finite dimensional
(see the Definition \ref{def of nilp conn} 
(\ref{connction finite dimensional})),
there exists an etale
covering $U\to X$ and a map
$$
\nabla:M \to A^1\otimes M,
$$
where $A^{\bullet}=t(\bold F_p(U/X))$. Let $End(M)^{nil}$
be the set of nilpotent endomorphism with respect to the filtration $F$.
The map $\nabla$ defines an element of 
$A^1\otimes End(M)^{nil}=\bold F_p(U\times_X U)\otimes End(M)^{nil}$,
which is also denoted by $\nabla$.
For an element $\varphi\in A^p\otimes End(M)^{nil}$ and
$\psi\in A^q\otimes End(M)^{nil}$, the composite $\varphi\circ \psi$
is defined by the following composite map
\begin{align*}
A^{p}\otimes End(M)^{nil} \otimes A^q \otimes End(M)^{nil}
&\to A^{p+q} \otimes End(M)^{nil}: \\
\omega\otimes \alpha \otimes \eta \otimes \beta &\mapsto 
(\omega\cup \eta)\otimes (\varphi\circ \psi).
\end{align*}
By the integrability condition, we have
$$
pr_{01}^*(\nabla)\circ pr_{12}^*(\nabla)=
pr_{12}^*(\nabla)-pr_{02}^*(\nabla)+pr_{01}^*(\nabla).
$$
We construct a descent data for $M_U$.
Via the isomorphism,
$$
\bold F_p(U\times_X U)\otimes End(M)\simeq
End(M_{U\times_XU}),
$$
the element $\rho=1-\nabla$ in $Aut(M_{U\times_XU})$ gives an
isomorphism $\rho:pr_0^* M_U\to pr_1^* M_U$.
Since
\begin{align*}
(pr_{01}^*(\rho))\circ
(pr_{12}^*(\rho))&=
(pr_{01}^*(1-\nabla))\circ
(pr_{12}^*(1-\nabla)) \\
&=1-pr_{01}^*(\nabla)-pr_{12}^*(\nabla)+
pr_{01}^*(\nabla)\circ pr_{12}^*(\nabla)
\\
&=
1-pr_{02}^*(\nabla)
=pr_{02}^*(\rho).
\end{align*}
Therefore $\rho=1-\nabla$ satisfies the 1-cycle condition and it defines a
descent data for $M_U$.
We see that a homomorphism of connections defines a morphism of
the descended sheaves. This gives an equivalence of the categories.
\end{proof}
By the theory of Tannakian categories, we have the following theorem.
\begin{corollary}
\label{isomorphism as a coalgebra}
The space $H^0(B_{red}(\check{C}(I,AS(\Cal W)),\epsilon))^*$ is isomorphic to 
$\bold F_p[[\pi_1(X,\bar x)]]$, where
$$
\bold F_p[[\pi_1(X,\bar x)]]=
\underset{\underset{\ot}n}\lim\  
\bold F_p[\pi_1(X,\bar x)]/I^n
$$
as algebras. 
\end{corollary}

\subsection{A variant for geometric base points}

Let $X$ be a connected separated scheme of finite type over $\bold F_p$
and $\Cal W=\{W_i\}$ be 
an affine covering of $X$. Let $\bar x:Spec(\overline{\bold F_p})\to X$ 
be an $\overline{\bold F_p}$-valued geoemetric point of $X$. 
By choosing a morphism $Spec(\overline{\bold F_p})\to W_i$ for some $i$,
we have a homomorphim $\tilde\epsilon:AS(W_i)\to AS(\overline{\bold F_p})$
and the induced DGA homomorphism
$\check{C}(I,AS(\Cal W))\to AS(\overline{\bold F_p})$
is also denoted as $\tilde \epsilon$.
Since the natural map $\bold F_p \to AS(\overline{\bold F_p})$
is a quasi-isomorphism, we can consider the bar complex
of $\check{C}(I,AS(\Cal W))$
for an augmentation map $\tilde\epsilon$.
We can consider the same daigram (\ref{diagram 1}) except for
$\epsilon$.
As a consequence, we have the following proposition.
\begin{proposition}
The reduced bar complex
$
B(\check{C}(I,AS(\Cal W)),\tilde\epsilon)
$
is quasi-isomorphic to the bar complex 
$\underset{U}\lim\ B(t(\bold F_p(U/X)),\bar\epsilon)$,
and
$H^0(B(\check{C}(I,AS(\Cal W)),\tilde\epsilon))$
is isomorphic to
$\underset{U}\lim\ H^0(B(t(\bold F_p(U/X)),\bar\epsilon))$
as a coalgebra.
\end{proposition}
\subsection{Universal pro-$p$ covering of schemes}

In this section, we give a description of the universal object
which gives an equivalece of the category
of $H^0(B(AS(\Cal W/X),\epsilon))$-comodules and 
that of $\bold F_p$ local systems on a connected 
$\bold F_p$-scheme $X$ of finite type. 
We set $A^{\bullet}=AS(\Cal W/X)$. 

Let $H$ be a coalgebra and $L$ and $M$ be a right and left comodules over $H$.
The coaction is denoted as
\begin{align*}
\ _L\Delta :L\to L\otimes H, \quad \Delta_M:M\to H\otimes M.
\end{align*}
The cotensor product of $cotor_H(L,M)$ is defined by
the kernel of the map
$$
L \otimes M \to L\otimes H\otimes M:l\otimes m\mapsto 
\ _L\Delta(l)\otimes m- l\otimes  \Delta_M(m).
$$
Let
$$
M=(M\overset{\nabla}\longrightarrow A^1\otimes M)
$$ 
be an $A$-connection
with a right $H^0(B(A,\epsilon))$ comodule structure
$$
\ _M\Delta:M\to M\otimes H^0(B(A,\epsilon))
$$
such that
$$
\begin{matrix}
M&\to &M\otimes H^0(B(A,\epsilon)) \\
\nabla_M\downarrow\phantom{***} & & \downarrow \nabla_M\otimes 1 \\ 
A^1\otimes M & \to & A^1\otimes M\otimes H^0(B(A,\epsilon))
\end{matrix}
$$
is commutative. Then for a left  
$H^0(B(A,\epsilon))$ comodule $F$, $cotor_H(M,F)$
becomes an $A$-connection since the following diagram is commutative:
$$
\begin{matrix}
M\otimes F
&\to & M\otimes H^0(B(A,\epsilon)) \otimes F \\
\nabla_M\otimes 1 \downarrow \phantom{*****} & &
\phantom{*****} 
\downarrow \nabla_M\otimes 1 \otimes 1\\
A^1\otimes M\otimes F
&\to & 
A^1\otimes M\otimes H^0(B(A,\epsilon)) \otimes F. \\
\end{matrix}
$$
We apply the defnition of $cotor_H(M,F)$ if $M$ is an inductive
limit of (finite dimensional) integrable nilpotent connection
with a left $H^0(B(A,\epsilon))$-coaction.

\begin{definition}[Universal connection]
The pair $(M,\Delta_M)$ is called the universal connection if
the functor
$$
(H-comod)\to (HINC_A):F \mapsto cotor_H(M,F)
$$
is an equivalence of categories.
\end{definition}
It is easy to see that the universal connection is unique
up to isomorphism if it exists.
In this section, we construct the universal connection.

Since $X$ is a connected $\bold F_p$ scheme, 
$A^{\bullet}$ is a connected DGA over $\bold F_p$.
We choose ${A'}^1\subset A^1$ such that ${A'}^1\oplus dA^0\simeq A^1$.
Then 
$$
{A'}=\bold F_p\oplus {A'}^1\oplus\bigoplus_{i\geq 2}A^i
$$
is a sub-DGA of $A$ and is quasi-isomorphic to $A$.
Also $B_{red}(A',\epsilon)^i=0$ for $i<0$ and we have
$i:H^0(B_{red}(A',\epsilon))\subset B_{red}(A',\epsilon)$.
We consider the comultiplication
$$
H^0(B_{red}(A',\epsilon))\to
H^0(B_{red}(A',\epsilon))\otimes 
H^0(B_{red}(A',\epsilon))
$$
defined in (\ref{comult on h0}).
By composing the inclusion $i$, and the projection 
$B_{red}(A',\epsilon)\to A^1$, we have a map
$$
\nabla_B:H^0(B_{red}(A',\epsilon))\to
B_{red}(A',\epsilon)\otimes 
H^0(B_{red}(A',\epsilon))\to
A^1\otimes 
H^0(B_{red}(A',\epsilon)).
$$
Then the pair $(H^0(B_{red}(A',\epsilon)),\nabla_B)$
is an integrable nilpotent $A$-connection.
By Theorem 7.6 of \cite{T}, we have the following proposition. 
\begin{proposition}
The connection $(\nabla_B,H^0(B_{red}(A',\epsilon)))$
is the universal connection.
\end{proposition}
\begin{example}[Affine case]
Let $R$ be an integral domain over $\bold F_p$.
We choose $V$ such that $Im(x\mapsto x^p-x)\oplus V\simeq R$.
Then we have
$
H^0(B_{red}(A',\epsilon))\simeq \oplus_{i=0}^\infty V^{\otimes i}
$
and the universal connection is given as 
$$
\nabla:H^0(B_{red}(A',\epsilon)) \to V\otimes H^0(B_{red}(A',\epsilon)):
[a_1\mid a_2\mid \cdots \mid a_n]\mapsto
a_1\otimes [a_2\mid \cdots \mid a_n].
$$
The corresponding etale sheaf $F_{univ}$ is given by
$$
F(R)_{univ}= Ker\Big[
R\otimes H^0(B_{red}(A',\epsilon))
\overset{
d_{AS}\otimes 1-(\mu\otimes 1)(1\otimes\nabla)
}
\longrightarrow R\otimes H^0(B_{red}(A',\epsilon))\Big].
$$
for an etale $A$-algebra $R$. Here the differential
$d_{AS}$ is given by $d_{AS}(r)=r^p-r$.
\end{example}
\section{Homotopy shuffle product for strict cosimplicial DGA}
\label{Shuffle}
\subsection{Homotopy shuffle system}
In Corollary \ref{isomorphism as a coalgebra}, 
we showed the isomorphism
\begin{equation}
\label{isom pi1 and bar}
H^0(B_{red}(\check{C}(I,AS(\Cal W)),\epsilon)^* \simeq
\bold F_p[[\pi_1(X,\bar x)]]
\end{equation}
as coalgebras. 
The algebra $\bold F_p[[\pi_1(X,\bar x)]]$ has a coproduct
and
the $\bold F_p$-completion of $\pi_1(X,\bar x)$
(actually it is isomorphic to the pro-$p$ 
completion of $\pi_1(X,\bar
x)$, see Appendix.)
is defined to be the set of
group like elements in 
$\bold F_p[[\pi_1(X,\bar x)]]$.
The coproduct on $\bold F_p[[\pi_1(X,\bar x)]]$
corresponds to a product on
$H^0(B_{red}(\check{C}(I,AS(\Cal W)),\epsilon)$.
In the theory of real homolopy type, the cohomology of bar complex
the $C^{\infty}$ differential forms is equipped with
a product structure obtained from the shuffle product.

If the characeristic of the base field $\bold k$ is zero,
the associate simple complex of a strict cosimplicial commutative
DGA has a natural structure of commutative DGA by Thom-Whitney construction.
In our case, $\operatorname{char}(\bold k)=p>0$ and the Artin-Schreier
DGA is not graded commutative and as a consequence,
the shuffle product is not available as it is. Theough Artin-Schreier DGA
is not graded commutative, it is graded commutative up to 
homotopy. In this section, we discuss higher homotopy for
commutativity, which is necessary to define the homotopy
shuffle product on the bar complex.

Let $I$ be a finite ordered set and $A$ be a strict cosimplicial
graded commutative DGA indexed by $I$.
In this section, we construct a homotopy shuffle product 
$$
\bold B_{red}(\check{C}(I,A)\otimes,\epsilon)\otimes
\bold B_{red}(\check{C}(I,A)\otimes,\epsilon)\to
\bold B_{red}(\check{C}(I,A)\otimes,\epsilon),
$$
which induces a natural coproduct on $\bold F_p[[\pi_1(X,\bar x)]]$
via the linear isomorphism (\ref{isom pi1 and bar}).

Let $A$ be a complex. Let $\sigma$ be an element of $\frak S_n$.
A homogeneous element 
$a_1e_1\otimes \cdots \otimes a_ne_n$
of the tensor product 
$(A[1])^{\otimes m}=A[1]\otimes \cdots \otimes A[1]$
of $n$-copies of $A[1]$ goes to 
$\pm a_{\sigma(1)}e_{\sigma(1)}\otimes \cdots \otimes
a_{\sigma(n)}e_{\sigma(n)}$ by the rule mentioned in
\S \ref{convention}
. Here $e_i$ is the canoncial generator of $\bold k[1]$
of degree $-1$.
The element $\sigma$ also acts on $A^{\otimes n}[n]$ 
via the identification
\begin{equation}
\label{bar twist identification}
(A[1])^{\otimes n} \simeq A^{\otimes n}[n]:a_1e_1\otimes \cdots
\otimes a_ne_n \mapsto (a_1\otimes \cdots \otimes a_n) e_1\cdots e_n
=(a_1\otimes \cdots \otimes a_n) e^n.
\end{equation}
For example, let $a_1$ and $a_2$ be elements
of degree $\alpha_1$ and $\alpha_2$ in $A$, respectively. 
Via this identification, we have
$$
a_1e_1\otimes a_2e_2 = (-1)^{\alpha_2}(a_1\otimes a_2)e_1e_2
$$
and the permutation of first and the second components becomes the
following map.
\begin{align*}
&(a_1\otimes a_2)e^2=
(a_1\otimes a_2)e_1e_2=
 (-1)^{\alpha_2}a_1e_1\otimes a_2e_2
\mapsto \\
& (-1)^{\alpha_2+(\alpha_1+1)(\alpha_2+1)}
a_2e_2 \otimes a_1e_1 
=
(-1)^{\alpha_2+(\alpha_1+1)(\alpha_2+1)+\alpha_1}
a_2\otimes a_1e_2e_1 \\
=&
(-1)^{\alpha_1\alpha_2+1}
a_2\otimes a_1e^2.
\end{align*}
Let $A$ be a DGA and $\mu:A\otimes A\to A$ be the multiplication
morphism. The morphism $\mu$ induces the morphism
$A[p]\otimes A[q]\to A[p+q]$ by the composite
defined in \S \ref{convention} (\ref{extention rule}).
If $A$ is graded commutative, then
$\mu + \mu\circ \sigma:A[1]\otimes A[1]\to A[2]$ is the zero map.

\begin{definition}[Homotopy shuffle system]
\label{def of homotopy shuffle system}
Let $A^{\bullet}$ be an associative DGA. A homotopy shuffle system
consists of the system of degree preserving $\bold k$ linear map
$$
h^{l,m}:A^{\otimes l}[l]\otimes A^{\otimes m}[m] \to A[1]
$$
for $l\geq 0,m\geq 0,l+m>0$ with the following properties.
The map $A^{\otimes l}[l+1]\otimes A^{\otimes m}[m] \to A[2]$
and $A^{\otimes l}[l]\otimes A^{\otimes m}[m+1] \to A[2]$
induced by $h^{l,m}$ are also denoted as $h^{l,m}$
using the rule (\ref{extention rule}).

\begin{enumerate}
\item
$h^{1,0}=h^{0,1}=1_A$, and $h^{0,p}=h^{p,0}=0$ for $p>1$.
\item
\label{second axion of homotopy shffle system}
Let $\mu:A\otimes A\to A$ be the multiplication for the DGA.
The induced map $A[1]\otimes A[1]\to A[2]$ is also denoted as
$\mu$. Let $\sigma$ be the switching homomorphism defined in 
(\ref{switching homomorphism}).
Then we have
$$
\mu+\mu\circ\sigma=-dh^{1,1}+h^{1,1}d:A[1]\otimes A[1] \to A[2].
$$
\item
\label{3rd condition for homotopy shuffle}
Let $l+m>2,l>0$ and $m>0$.
We define the following composite linear map $\alpha$.
\begin{align}
\label{homotopy shuffle 1st term}
\alpha: A^{\otimes l}[l]\otimes A^{\otimes m}[m]
&
\overset{\sigma}\longrightarrow
\underset{
\substack{l'+l''=l,\\ 
m'+m''=m,\\ 
l'+ m'\geq 1\\
l''+ m''\geq 1
}}
\bigoplus 
(A^{\otimes l'}[l']\otimes A^{\otimes m'}[m'])\otimes
(A^{\otimes l''}[l'']\otimes A^{\otimes m''}[m'']) \\
&
\nonumber
\overset{\sum (h^{l',m'}\otimes h^{l'',m''})}
\longrightarrow
A[1] \otimes A[1] \\
&
\nonumber
\overset{\mu}\longrightarrow A[2]
\end{align}
Here the morphism $\sigma$ is given by the 
permutation of components. We define the following 
linear maps $\beta$ and $\gamma$.

\begin{align}
\label{homotopy shuffle 2nd term}
\beta:A^{\otimes l}[l]\otimes A^{\otimes m}[m] & 
\overset{-d_B\otimes 1}\longrightarrow
A^{\otimes (l-1)}[l]\otimes A^{\otimes m}[m] \\
\nonumber
& \overset{h^{l-1,m}}\longrightarrow
A[2]
\end{align}
\begin{align}
\label{homotopy shuffle 3rd term}
\gamma:A^{\otimes l}[l]\otimes A^{\otimes m}[m] & 
\overset{-1\otimes d_B}\longrightarrow 
A^{\otimes l}[l]\otimes A^{\otimes (m-1)}[m] \\
\nonumber
& \overset{h^{l,m-1}}\longrightarrow 
A[2]
\end{align}
Then $h^{l,m}$ satisfies the relation:
$$
-dh^{l,m}+h^{l,m}d=\alpha+\beta+\gamma:
A^{\otimes l}[l]\otimes A^{\otimes m}[m]  \longrightarrow
A[2].
$$
\end{enumerate}
\end{definition}

For the augmentation ideal $I$, a homotopy shuffle system 
$h^{l,m}:I^{\otimes l}[l]\otimes I^{\otimes m}[m]\to I[1]$
is defined similarly.
In the rest of this subsection, 
we define the homotopy shuffle product 
$$
B_{red}(A,\epsilon)\otimes B_{red}(A,\epsilon)\to B_{red}(A,\epsilon)
$$
by assuming the existence of a homotopy shuffle system
$\{h^{l,m}\}_{l,m}$ for $I$.
We define the set $S_{l,m}^{(k)}$ by
\begin{align*}
S_{l,m}^{(k)}=& \{
f:[1,l]\cup[\bar 1, \bar m] \to [1,l+m-k]\mid
 (1) \text{ $f$ is surjecitive, } \\
& (2)\  f(i)\leq f(j) \text{ for }i\leq j, \text{ and } 
 (3)\  f(\bar i)\leq f(\bar j) \text{ for }i\leq j 
\}.
\end{align*}
For $f\in S_{l,m}^{(k)}$ and $p\in [1, l+m-k]$, 
we define $S(f,p)=f^{-1}(p)$
and $S'(f,p)=S(f,p)\cap [1,l]$,
$S''(f,p)=S(f,p)\cap [\bar 1,\bar m]$.
The cardinality of $S(f,p),S'(f,p)$ and $S''(f,p)$
is denoted by $s(f,p),s'(f,p)$ and $s''(f,p)$, respectively.
We define a linear map
$H(f):I^{\otimes l}[l]\otimes I^{\otimes m}[m] \to I^{l+m-k}[l+m-k]$ by
the composite

\begin{align*}
&I^{\otimes l}[l]\otimes I^{\otimes m}[m] \\
\overset{\sigma}\to & 
\Big(I^{s'(f,1)}[s'(f,1)]
\otimes I^{s''(f,1)}[s''(f,1)]\Big)
\otimes \cdots  \\
&\otimes 
\Big(I^{s'(f,l+m-k)}[s'(f,l+m-k)]
\otimes I^{s''(f,l+m-k)}[s''(f,l+m-k)]\Big)\\
\overset{
\theta
}
\longrightarrow & I[1] 
\otimes \cdots \otimes 
I[1] \\
\simeq &
I^{\otimes l+m-k}[l+m-k],
\end{align*}
where $\theta=h^{s'(f,1),s''(f,1)}\otimes\cdots \otimes 
h^{s'(f,l+m-k),s''(f,l+m-k)}$ and
$\sigma$ is the permutation homomorphism.
We define a linear map $H_{l,m,k}$ by
$$
H_{l,m,k}=\sum_{f\in S_{l,m}^{(k)}} H(f):
I^{\otimes l}[l]\otimes I^{\otimes m}[m]\to
I^{\otimes (l+m-k)}[l+m-k].
$$
\begin{proposition}
The linear map
$H=\sum_{l,m,n}H_{l,m,n}$
defines a homomorphism of complexes
$$
B(A,\epsilon) \otimes B(A,\epsilon)\to B(A,\epsilon).
$$
\end{proposition}
\begin{proof}
The differential of the bar complex is the sum of inner differential
$d_I:I^{\otimes n}[n]\to I^{\otimes n}[n+1]$ and bar differential
$d_B:I^{\otimes n}[n]\to I^{\otimes (n-1)}[n]$, where
\begin{align*}
&d_I=\sum_{i=1}^nd_{I,i}, \quad
d_{I,i}=1\otimes \cdots \otimes d\otimes \cdots \otimes 1, \\
&d_B=\sum_{i=1}^{n-1}d_{B,i},\quad
d_{B,i}(x_1\otimes \cdots \otimes x_n)=
x_1\otimes \cdots \otimes (x_i\cup x_{i+1})\otimes \cdots \otimes x_n.
\end{align*}
Here we used the identification (\ref{bar twist identification})
and the definition of sign (\ref{definition of reduced bar differential}).

We consider the following diagram
$$
\begin{matrix}
I^{\otimes l}[l]\otimes I^{\otimes m}[m] & 
\overset{H_{l,m,k}}\longrightarrow & I^{\otimes (l+m-k)}[l+m-k] \\
d_B\otimes 1 + 1\otimes d_B
\downarrow & & \downarrow d_B \\
I^{\otimes (l-1)}[l]\otimes I^{\otimes m}[m]\oplus
I^{\otimes l}[l]\otimes I^{\otimes (m-1)}[m] & 
\overset{H_{l-1,m,k}+H_{l,m-1,k}}\longrightarrow &
I^{\otimes (l+m-k-1)}[l+m-k].
\end{matrix}
$$
It is enough to show that
\begin{align}
\label{homotopy shuffle l,m,k part}
&d_B\circ H_{l,m,n}-
H_{l-1,m,k}
\circ
(d_B\otimes 1) -H_{l,m-1,k}\circ(1\otimes d_B) \\
\nonumber
=&
-d_IH_{l,m,k+1}+H_{l,m,k+1}d_I.
\end{align}
Let $\bar f\in S_{l,m}^{(k+1)}$. 
For $i\in [1, l+m-k-1]$, we define a surjective 
non decreasing map 
$\sigma_i:[l+m-k]\to [l+m-k-1]$
by setting $\sigma_i(i)=\sigma_i(i+1)$.
We set 
$$
S_{l,m}^{(k)}(\bar f)=
\{ (i, f)\in [1,l+m-k-1]\times S_{l,m}^{(k)} \mid
\sigma_i\circ f=\bar f\}.
$$
We define $T_{l-1,m}^{(k)}(\bar f)$ and ${T'}_{l,m-1}^{(k)}(\bar f)$ 
by
\begin{align*}
& T_{l-1,m}^{(k)}(\bar f)=\{
(j,g)\in [1,l-1]\times S_{l-1,m}^{(k)}\mid
g \circ (\sigma_j\times 1)=\bar f
\}, \\
& {T'}_{l-1,m}^{(k)}(\bar f)=\{
(j',g')\in [\overline{1},\overline{m-1}]\times S_{l,m-1}^{(k)}\mid
g' \circ (1\times {\sigma}_{j'})=\bar f
\}.
\end{align*}
To show the equality (\ref{homotopy shuffle l,m,k part}),
it is enough to show the following identity
by considering $\bar f$-component.
\begin{align}
\label{bar f part}
&\sum_{(i,f)\in S_{l,m}^{(k)}(\bar f)}d_{B,i}\circ H(f) \\
\nonumber
&-
\sum_{(j,g)\in T_{l-1,m}^{(k)}(\bar f)}H(g)
\circ
(d_{B,j} \otimes 1)
-
\sum_{(j',g')\in {T'}_{l,m-1}^{(k)}(\bar f)}
H(g')\circ (1\otimes d_{B,j'}) \\
\nonumber
=&
-d_IH(\bar f)+H(\bar f)d_I
\end{align}
We fix $\bar f$ and $i$. The set
\begin{align*}
S_{l,m}^{(k)}(\bar f,i)=&
\{f\mid (i,f)\in S(\bar f)\} \\
=& \{(P',P'')\mid 
 \text{$P'\coprod P''=S(\bar f,i)$ is a non empty partition} \\
& \text{such that $p'<p''$ for $p'\in P'$ and $p''\in P''$}
\} 
\end{align*}
and
\begin{align*}
& T_{l-1,m}^{(k)}(\bar f,i)=\{
(j,g)\in [1,l-1]\times S_{l-1,m}^{(k)}\mid
g \circ (\sigma_j\times 1)=\bar f,\bar f(j)=i)
\}, \\
& {T'}_{l-1,m}^{(k)}(\bar f,i)=\{
(j',g')\in [\overline{1},\overline{m-1}]\times S_{l,m-1}^{(k)}\mid
g' \circ (1\times {\sigma}_{j'})=\bar f,\bar f(j')=i)
\}. 
\end{align*}
Then by the condition for homotopy suffle system, we have
\begin{align*}
&\sum_{(i,f)\in S_{l,m}^{(k)}(\bar f,i)}d_{B,i}\circ H(f) \\
&-
\sum_{(j,g)\in T_{l-1,m}^{(k)}(\bar f,i)}H(g)
\circ
(d_{B,j} \otimes 1)
-
\sum_{(j',g')\in {T'}_{l,m-1}^{(k)}(\bar f,i)}
H(g')\circ (1\otimes d_{B,j'}) \\
=&
-d_{I,i}H(\bar f)+\sum_{j\in S(\bar f,i)}H(\bar f)d_{I,j}.
\end{align*}
By taking the summation over $i$, we have the equality 
(\ref{bar f part}).
\end{proof}

\subsection{Construction of homotopy shuffle system}

In this section, we construct a homotopy shuffle system for a total DGA of
a strict cosimplicial graded commutative DGA's
over the set $I=[0,n]$. If $n=0$, since the total
DGA $t(A)=A_0$ is graded commutative, 
we define 
$h^{l,m}=0$ for $l+m>1$. 

We define $h^{l,m}$ by the induction on $n$.

\subsubsection{The case for $n=1$.}
The total DGA $A$ is written as $A_0\oplus A_1 \oplus A_{0,1}[-1]$
and the multiplication can be written as
\begin{align*}
&(a_0+a_1+a_{01})\cup (b_0+b_1+b_{01})\\=&
(a_0\cdot b_0) + (a_1\cdot b_1) + 
(i_0(a_0)\cdot b_{01}+a_{01}\cdot i_1(b_1)) \\
=&(a_0\cdot b_0) + (a_1\cdot b_1) + 
(i_0(a_0)\cdot (b_{01}e)+(-1)^{\deg(b_1)}(a_{01}e)\cdot i_1(b_1))e^{-1}.
\end{align*}
for $a_0,b_0 \in A_0,a_1,b_1 \in A_1,a_{01},
b_{01} \in A_{01}[-1]$.
Here we used 
the DGA homomorphism $i_0:A_0 \to A_{01}, i_1:A_1\to A_{01}$
and
``$\cdot$'' for the multiplication of $A_0,A_1$ and
$A_{01}$.

We set $h^{1,1}(a\otimes b)
=-(a_{01}e)\cdot (b_{01}e)\in A_{01}\subset A[1]$ 
for elements $a=(a_0+a_1+a_{01})e$ and
$b=(b_0+b_1+b_{01})e$ in $A[1]$,
and
$h^{l,m}=0$ for $l+m>2$.
We check the conditions for a homotopy shuffle system.
To avoid the sign complexity, we only check for
degree zero elements.
\begin{proposition}
Let $x,y$ be degree zero elements in $A[1]$.
We have the following formula.
\begin{align}
\label{step1 1st}
& x\cup y+y\cup x
=-d_Ah^{1,1}(x\otimes y)+h^{1,1}(d_{A\otimes A}(x\otimes y)) \\
\label{step1 2nd}
&h^{1,1}((x_1\cup x_2)\otimes y)=
x_1\cup h^{1,1}(x_2\otimes y)+
h^{1,1}(x_1\otimes y)\cup x_2 \\
\label{step1 3rd}
&h^{1,1}(x\otimes (y_1\cup y_2))=
y_1\cup h^{1,1}(x\otimes y_1)+
h^{1,1}(x\otimes y_1)\cup y_2
\end{align}
\end{proposition}
\begin{proof}
We set $x=(x_0+x_1+x_{01})e$ and $y=(y_0+y_1+y_{01})e$.
Since $A_0,A_1$ are graded commutative, to prove the equality 
(\ref{step1 1st}), it is enough to show the following equality.
$$
-\Big[(i_0(x_0)\cdot y_{01}+x_{01}\cdot i_1(y_1))
+(i_0(y_0)\cdot x_{01}+y_{01}\cdot i_1(x_1))\Big]e^2
=(-d_Ah^{1,1}+h^{1,1}d_{A\otimes A})(x\otimes y)
$$
The left hand side is equal to
$$
-\Big[i_0(x_0)\cdot y_{01}e-x_{01}e\cdot i_1(y_1)
+i_0(y_0)\cdot x_{01}e-y_{01}e\cdot i_1(x_1)\Big]e.
$$
Let $d_0,d_1$ and $d_{01}$ be the differentials of
$A_0,A_1$ and $A_{01}$, respectively.
Using expression of (\ref{differnetial as degree zero map}), 
we have
$d_A(x)=(d_0x_0+d_1x_1+d_{01}x_{01})e^2+(i_0x_0-i_1x_1)e$ and
by the definition of extension rule defined in (\ref{extention rule}),
the right hand side of the 
above equality is equal to 
\begin{align*}
& 
d_{01}(x_{01}e\cdot y_{01}e)e \\
&+h^{1,1}(((d_0x_0+d_1x_1+d_{01}x_{01})e^2+(i_0(x_0)-i_1(x_1))e)
\otimes
(y_0+y_1+y_{01})e) \\
&+h^{1,1}(
(x_0+x_1+x_{01})e\otimes
((d_0y_0+d_1y_1+d_{01}y_{01})e^2+(i_0(y_0)-i_1(y_1))e) \\
=
&
\big[(i_1(x_1)-i_0(x_0))\cdot (y_{01}e)\Big]e+\Big[(x_{01}e)\cdot
(i_1(y_1)-i_0(y_0))\Big] e \\
=
&
\big[y_{01}e\cdot i_1(x_1)-i_0(x_0)\cdot y_{01}e-i_0(y_0)\cdot(x_{01}e)
+x_{01}e\cdot i_1(y_1)\Big] e
\end{align*}
Thus we have (\ref{step1 1st}).

As for the equality (\ref{step1 2nd}), we have
$$
h^{1,1}((x^{(1)}\cup x^{(2)})\otimes y)=
-(i_0(x^{(1)}_0)\cdot x^{(2)}_{01}+ x^{(1)}_{01}\cdot i_1(x^{(2)}_1))
\cdot y_{01}
$$
and
\begin{align*}
&x^{(1)}\cup h^{1,1}(x^{(2)}\otimes y)+ 
h^{1,1}(x^{(1)}\otimes y) \cup x^{(2)} \\
=&
-i_0(x^{(1)}_0)\cdot x^{(2)}_{01}\cdot y_{01}
-x^{(1)}_{01}\cdot y_{01}\cdot i_1(x^{(2)}_{1}).
\end{align*}
The equality (\ref{step1 2nd}) is similar.
\end{proof}
We will check the condition of 
Definition \ref{def of homotopy shuffle system}
(\ref{3rd condition for homotopy shuffle})
for the homotopy shuffle system.
In the case of $l+m=3$, the condition is nothing by the equality
(\ref{step1 2nd}), (\ref{step1 3rd}). If $l+m>3$, then
the homomorphisms $\alpha$ is equal to zero. 
In fact, unless
$l'=l''=m'=m''=1$, 
then either $l',l'',m'$ or $m''$ is greater than one.
The term for $l'=l''=m'=m''=1$ is also zero by the relation
$$
h^{1,1}(x^{(1)}\otimes y^{(1)})\cup
h^{1,1}(x^{(2)}\otimes y^{(2)})=0.
$$
$\beta$ and $\gamma$ are also zero.

\subsubsection{The case for $n>1$}
Let $B$ is a commutative DGA and $A_{\bullet}$ a strict cosimplicial 
graded commutative DGA.
The map $A_{k}\to A_{jk}$ is denoted as $i_{jk,k}$.
A set of homomorphisms of DGA's $i_k:B\to A_{k}$ is called central 
if $i_{jk,k}\circ i_k=i_{jk,j}\circ i_j$.
The map $H:\Cal A\otimes \Cal A \to \Cal A[-1]$ is called central if
for any central homomomorphism $i:B\to \Cal A$,
$H(i(b)\otimes a)=H(a\otimes i(b))=0$.

We construct $h^{l,m}$ by the induction of $n$.
Let $A_{\bullet}$ be a strict cosimplicial of graded commutative
DGA. Let $\Cal A_0$ be the strict Cech complex of $\{A_I\}_{I\subset [1,n-1]}$.
It is easy to see that the system 
$\{A_{I\cup n}\}_{I\subset [1,n-1]}$ also becomes a strict cosimplicial 
graded commutative DGA and the homomorphism $A_I \to A_{I\cup n}$
defines a homomorphism of strict cosimplicial DGA. The Cech complex of
$\{A_{I\cup n}\}_{I\subset [1,n-1]}$ is denoted as $\Cal A_{01}$
and the homomorphism $\Cal A_0\to \Cal A_{01}$ is denoted as $i_0$.
We set $\Cal A_1=A_n$.
The homomorphisms $A_n \to A_{I\cup n}$ defines a homomorphism of
DAG $\Cal A_1 \to \Cal A_{01}$, which is denoted as $i_1$.
Thus we have a strict cosimplicial DGA
$$
\Cal A_0\oplus \Cal A_1 \to \Cal A_{01}.
$$
The total DGA $\Cal A=\Cal A_0\oplus \Cal A_0\oplus \Cal A_{01}[-1]$ 
of the above diagram is equal to 
the Cech DGA of $A_{\bullet}$.
By the inductive construction, we already have $\bold k$ linear maps
$$
h^{l,m}_{\star}=h^{l,m}_{\Cal A_{\star}}:
\Cal A_{\star}^{\otimes l}[l]\otimes
\Cal A_{\star}^{\otimes m}[m]\to
\Cal A_{\star}[1]
$$
for $\star=0,1,01$ compatible with $i_0$ and $i_1$. 
(Note that $h^{1,1}_1=0$, since it is strict
cosimplicial on the set $\{n\}$.)
We construct 
central $\bold k$ linear map
$$
h^{l,m}_{\Cal A}:\Cal A^{\otimes l}[l]\otimes \Cal A^{\otimes m}[m]\to 
\Cal A[1].
$$

{\bf The case $(l,m)=(1,1)$.}
We construct $h^{1,1}_{\Cal A}$
so that the relation Definition \ref{def of homotopy shuffle system}
(\ref{second axion of homotopy shffle system}) holds.
Let $x=(x_0+x_1+x_{01})e,y=(y_0+y_1+y_{01})e$ be elements in $\Cal A[1]$, where
$x_{0},y_{0}\in \Cal A_{0}$, $x_{1},y_{1}\in \Cal A_{1}$ 
and $x_{01},y_{01}\in \Cal A_{01}[-1]$.
Since the map of $i_1:A_n\to (\Cal A_{01})_k=A_{kn}$ is central
and $\Cal A_1$ is commutative,
We set $h_{\Cal A}^{1,1}:\Cal A[1]\otimes \Cal A[1]\to \Cal A[1]$ by
$$
h^{1,1}_{\Cal A}(x\otimes y)=h^{1,1}_0(x_0e\otimes y_0e)
-h^{1,1}_{01}(x_{01}e\otimes i_0(y_0)e)-(x_{01}e)\cdot (y_{01}e),
$$
and will show the equality:
$$
x\cup y+y\cup x =-d_{\Cal A}h^{1,1}_{\Cal A}(x\otimes y)+h^{1,1}_{\Cal A}
(d_{\Cal A\otimes \Cal A}(x\otimes y)).
$$
For simplicity, we assume that the degree of $x$ and $y$ is zero.
we have
\begin{align*}
&x\cup y+y\cup x \\
=& (x_0+x_1+x_{01})e\cup(y_0+y_1+y_{01})e+
(y_0+y_1+y_{01})e\cup(x_0+x_1+x_{01})e \\
=& -(x_0\cdot y_0+y_0\cdot x_0)e^2
+(-i_0(x_0)\cdot y_{01}e^2+x_{01}e\cdot i_1(y_1)e)\\
&+(-i_0(y_0)\cdot x_{01}e^2+y_{01}e\cdot i_1(x_1)e) \\
=& -(x_0\cdot y_0+y_0\cdot x_0)e^2
+(-i_0(x_0)+i_1(x_1))\cdot y_{01}e^2
-x_{01}e\cdot(i_0(y_0)-i_1(y_1))e \\
&+(x_{01}e\cdot i_0(y_0)e-i_0(y_0)\cdot x_{01}e^2)
\end{align*}
On the other hand,
\begin{align*}
d_{\Cal A\otimes \Cal A}(x\otimes y)=&
((d_0x_0+d_1x_1+d_{01}x_{01})e^2+(i_0(x_0)-i_1(x_1))e)
\otimes (y_0+y_1+y_{01})e \\
&+
((x_0+x_1+x_{01})e\otimes
((d_0y_0+d_1y_1+d_{01}y_{01})e^2+(i_0(y_0)-i_1(y_1))e).
\end{align*}
We have
\begin{align*}
&-d_{\Cal A}h^{1,1}_{\Cal A}(x\otimes y)
+h^{1,1}_{\Cal A}(d_{\Cal A\otimes \Cal A}(x\otimes y)) \\
=&-d_0h^{1,1}_{0}(x_0e\otimes y_0e)e-i_0(h^{1,1}_{0}(x_0e\otimes y_0e))
+d_{01}h^{1,1}_{01}(x_{01}e\otimes i_0(y_0)e)e
+d_{01}(x_{01}e\cdot y_{01}e)e \\
&+h^{1,1}_0(d_0(x_0e\otimes y_0e))e
-h^{1,1}_{01}((d_{01}x_{01}e^2+(i_0(x_0)-i_1(x_1))e)\otimes i_0(y_0)e) \\
&-h^{1,1}_{01}(x_{01}e\otimes i_0(d_0y_0)e^2) \\
&+(i_1(x_1)-i_0(x_0))\cdot y_{01}e^2
-(x_{01}e)\cdot(i_0(y_0)-i_1(y_1))e+d_{01}(x_{01}e\cdot y_{01}e)e 
\end{align*}
using the central assmption
$h^{1,1}_{01}(i_1(x_1)\otimes i_0(y_0))=0$ and
inductive hypotheses. Therefore the right hand side is equal to
\begin{align*}
& -(x_0\cdot y_0+y_0\cdot x_0)e^2
+(-i_0(x_0)+i_1(x_1))\cdot y_{01}e^2
-x_{01}e\cdot(i_0(y_0)-i_1(y_1))e \\
&+(x_{01}e\cdot i_0(y_0)e-i_0(y_0)\cdot x_{01}e^2).
\end{align*}
It is easy to see that $h^{1,1}_{\Cal A}$ is also central.

{\bf The case $(l,m)=(1,m)$ for $m>1$.}

Assume that $h^{(1,m)}_{\star}:
\Cal A_{\star}[1]\otimes \Cal (A_{\star}[1])^{\otimes m}\to
\Cal A_{\star}[1]$ is defined for 
$\star=0,1,01$.
Let $(x=x_0+x_1+x_{01})e, y^{(i)}=(y^{(i)}_{0}+y^{(i)}_{1}+y^{(i)}_{01})e$
be elements in $\Cal A[1]$,
where $x_{0},y^{(i)}_{0}\in \Cal A_{0}$,
$x_{1},y^{(i)}_{1}\in \Cal A_{1}$ and
$x_{01},y^{(i)}_{01}\in \Cal A_{01}[-1]$.
We define 
$h^{1,m}:\Cal A[1]\otimes (\Cal A[1])^{\otimes m}\to
\Cal A[1]$ by
\begin{align*}
&h^{1,m}(x\otimes y^{(1)}\otimes \cdots \otimes y^{(m)})\\
=&
h^{1,m}_0(x_0e\otimes y^{(1)}_0e\otimes \cdots \otimes y^{(m)}_0e)
-h^{1,m}_{01}(x_{01}e\otimes i_0(y^{(1)}_0)e\otimes \cdots \otimes 
i_0(y^{(m)}_0)e) \\
&-
\Big[h^{1,m-1}_{01}(x_{01}e\otimes i_0(y^{(1)}_0)e
\otimes \cdots \otimes 
i_0(y^{(m-1)}_0)e)\Big]\cdot (y_{01}^{(m)}e). \\
\end{align*}
The condition 
$$
\alpha+\beta+\gamma=
-d_{\Cal A[1]}h^{1,m}
+h^{1,m}d_{\Cal A[1]\otimes \Cal A[1]^{\otimes m}}
$$
in Definition \ref{def of homotopy shuffle system}
(\ref{3rd condition for homotopy shuffle})
is equal to
\begin{align}
\label{goal to show}
& (-1)^{\deg(x)\deg(y^{(1)})}y^{(1)}\cup h^{1,m-1}(x\otimes y^{(2)}
\otimes \cdots \otimes y^{(m)}) \\
\nonumber
&+
h^{1,m-1}(x\otimes y^{(1)}\otimes \cdots 
\otimes y^{(m-1)})\cup y^{(m)} \\
\nonumber
& -h^{1,m-1}(x\otimes (y^{(1)}\cup y^{(2)})
\otimes \cdots \otimes y^{(m)})-\cdots \\
\nonumber
&-
h^{1,m-1}(x\otimes y^{(1)}\otimes \cdots \otimes 
(y^{(m-1)}\cup y^{(m)})) \\
\nonumber
=&
(-d_{\Cal A}h^{1,m}
+h^{1,m}d_{\Cal A\otimes \Cal A^{\otimes m}})
(x\otimes y^{(1)}\otimes \cdots \otimes y^{(m)}).
\end{align}
We show this euqality by assuming that the degrees of $x$ and $y^{(i)}$
are zero.
The terms of the left hand side of (\ref{goal to show})
are as follows. 
\begin{align}
&y^{(1)}\cup h^{1,m-1}(x\otimes y^{(2)}\otimes 
\cdots \otimes y^{(m)}) \\
\nonumber
=&
y^{(1)}_0e\cdot h^{1,m-1}_0(x_0e\otimes y^{(2)}_0e\otimes 
\cdots \otimes y^{(m)}_0e) \\
\nonumber
&-i_0(y^{(1)}_0)e\cdot
h^{1,m-1}_{01}(x_{01}e\otimes i_0(y^{(2)}_0)e\otimes \cdots \otimes 
i_0(y^{(m)}_0)e) \\
\nonumber
&-
i_0(y^{(1)}_0)e\cdot
\Big[ h^{1,m-2}_{01}(x_{01}e\otimes i_0(y^{(2)}_0)e\otimes \cdots \otimes 
i_0(y^{(m-1)}_0)e)\cdot y_{01}^{(m)}e\Big],
\end{align}
\begin{align}
&h^{1,m-1}(x\otimes y^{(1)}\otimes \cdots \otimes y^{(m-1)})\cup y^{(m)} \\
\nonumber
=&
h^{1,m-1}_0(x_0e\otimes y^{(1)}_0e\otimes \cdots \otimes y^{(m-1)}_0e)
\cdot y^{(m)}_0e \\
\nonumber
&+
i_0(h^{1,m-1}_0(x_0e\otimes y^{(1)}_0e\otimes \cdots \otimes y^{(m-1)}_0e))
\cdot y^{(m)}_{01}e \\
\nonumber
&-
h^{1,m-1}_{01}(x_{01}e\otimes i_0(y^{(1)}_0)e\otimes \cdots \otimes 
i_0(y^{(m-1)}_0)e)\cdot i_1(y^{(m)}_1)e \\
\nonumber
&-
h^{1,m-2}_{01}(x_{01}e\otimes i_0(y^{(1)}_0)e\otimes \cdots \otimes 
i_0(y^{(m-2)}_0)e)\cdot y_{01}^{(m-1)}e\cdot i_1(y^{(m)}_1)e,
\end{align}
\begin{align}
&-h^{1,m-1}(x\otimes (y^{(1)}\cup y^{(2)})
\otimes \cdots \otimes y^{(m)})-\cdots \\
\nonumber
&-
h^{1,m-1}(x\otimes y^{(1)}\otimes \cdots \otimes 
(y^{(m-1)}\cup y^{(m)})) \\
\nonumber
=&-h^{1,m-1}_0(x_0e\otimes (y^{(1)}_0e\cdot y^{(2)}_0e)
\otimes \cdots \otimes y^{(m)}_0e)- \dots\\
\nonumber
&-
h^{1,m-1}_0(x_0e\otimes y^{(1)}_0e\otimes \cdots \otimes 
(y^{(m-1)}_0e\cdot y^{(m)}_0e)) \\
\nonumber
&+h^{1,m-1}_{01}(x_{01}e\otimes i_0(y^{(1)}_0e\cdot y^{(2)}_0e)
\otimes \cdots \otimes i_0(y^{(m)}_0)e)- \dots \\
\nonumber
&+h^{1,m-1}_{01}(x_{01}e\otimes i_0(y^{(1)}_0)e\otimes \cdots \otimes 
i_0(y^{(m-1)}_0e\cdot y^{(m)}_0e)) \\
\nonumber
&+h^{1,m-2}_{01}(x_{01}e\otimes i_0(y^{(1)}_0e\cdot y^{(2)}_0e)\otimes \cdots 
\otimes i_0(y^{(m-1)}_0)e)\cdot  y^{(m)}_{01}e- \dots\\
\nonumber
&+h^{1,m-2}_{01}(x_{01}e\otimes i_0(y^{(1)}_0)e\otimes \cdots 
\otimes i_0(y^{(m-2)}_0e\cdot y^{(m-1)}_0e))\cdot  y^{(m)}_{01}e\\
\nonumber
&
+h^{1,m-2}_{01}(x_{01}e\otimes i_0(y^{(1)}_0)e\otimes \cdots \otimes 
i_0(y^{(m-2)}_0)e) \\
\nonumber
&\quad \cdot
(i_0(y^{(m-1)}_0)e\cdot y^{(m)}_{01}e+y^{(m-1)}_{01}e
\cdot i_1(y^{(m)}_{1})e).
\end{align}
By the inductive hypotheses for $h^{(1,m)}_0,h^{(1,m)}_{01}$,
the left hand side of (\ref{goal to show}) is equal to
\begin{align*}
&(-d_0h^{1,m}_0+h^{1,m}_0d_0)(x_0e\otimes y^{(1)}_0e\otimes \cdots \otimes
y^{(m)}_0e) \\
&+(d_{01}h^{1,m}_{01}-h^{1,m}_{01}d_{01})
(x_{01}e\otimes i_0(y^{(1)}_0)e\otimes \cdots \otimes
i_0(y^{(m)}_0)e) \\
&+h^{1,m-1}_{01}(x_{01}e\otimes i_0(y^{(1)}_0)e
\otimes \cdots \otimes 
i_0(y^{(m-1)}_0)e)\cdot (i_1(y^{(m)}_1)e-i_0(y^{(m)}_0)e) \\
&+
(-d_{01}h^{1,m-1}_{01}+h^{1,m-1}_{01}d_{01})
(x_{01}e\otimes i_0(y^{(1)}_0)e\otimes \cdots \otimes 
i_0(y^{(m-1)}_0)e)\cdot y^{(m)}_{01}e \\
&-
h^{1,m-1}_{01}
((i_0(x_{0})e-i_1(x_1)e)\otimes i_0(y^{(1)}_0)e
\otimes \cdots \otimes 
i_0(y^{(m-1)}_0)e)\cdot y^{(m)}_{01}e 
\end{align*}
and this is equal to the right hand side of (\ref{goal to show}).

By the induction on $n$, we can prove the following proposition.
\begin{proposition}
\label{central linearity}
Let $\Cal A$ be the total DGA of a strict cosimplicial graded commutative
DGA's and $h^{(1,m)}$ be the homomorphism defined as above.
Let $B$ be a graded commutative DGA's and $j_k:B\to A_k$ be a family
of central morphisms of DGA's. Let $j:B\to \Cal A$ be the induced
morphisms of DGA's. Then we have
$$
h^{(1,m)}((x\cdot j(x'))\otimes y_1\otimes \cdots \otimes y_m).
=(-1)^{\deg j(x')\sum_k\deg y_k}
h^{(1,m)}(x\otimes y_1\otimes \cdots \otimes y_m)\cdot j(x').
$$
for $x,y_1, \dots, y_m \in \Cal A[1]$ and $x'\in B$.
\end{proposition}
{\bf The case $(l,m)$ for $l>1$}.

We define $h^{l,m}=0$ for $l>1$. To show the
condition Definition \ref{def of homotopy shuffle system}
(\ref{3rd condition for homotopy shuffle}), 
it is enough to consider the case where $l=2$.
In the case of $l=2$,
we prove this condition by the induction on $m$.
By the inductive hypotheses, we have
$$
h^{1,m}(x^{(1)}\otimes x^{(2)}\otimes y^{(1)}\otimes 
\cdots \otimes (y^{(i)}\cup
y^{(i+1)})\otimes \cdots \otimes y^{(m)})=0,
$$
and $\alpha+\beta+\gamma$ in
Definition \ref{def of homotopy shuffle system}
(\ref{3rd condition for homotopy shuffle})
is equal to
\begin{align*}
&x^{(1)}\cup h^{1,m}(x^{(2)}\otimes y^{(1)}\otimes \cdots \otimes y^{(m)}) \\
&+(-1)^{\deg x^{(2)} \deg y^{(1)}}
h^{1,1}(x^{(1)}\otimes y^{(1)})\cup
h^{1,m-1}(x^{(2)}\otimes y^{(2)}\otimes \cdots \otimes y^{(m)})\pm \cdots \\
&\pm
h^{1,m-1}(x^{(1)}\otimes y^{(1)}\otimes \cdots \otimes y^{(m-1)})\cup
h^{1,1}(x^{(2)}\otimes y^{(m)}) \\
&\pm
h^{1,m}(x^{(1)}\otimes y^{(1)}\otimes \cdots \otimes y^{(m)})\cup x^{(2)} \\
&-
h^{1,m}((x^{(1)}\cup x^{(2)})\otimes y^{(1)}\otimes \cdots \otimes y^{(m)}).
\end{align*}
Again we consider the case where the degrees of $x^{(i)}$ and $y^{(j)}$
are zero.
By the inductive definition of $h^{1,m}$, it is equal to
\begin{align*}
&x^{(1)}_0e \cdot h^{1,m}_0 (x^{(2)}_0e \otimes 
y^{(1)}_0e \otimes \cdots \otimes y^{(m)}_0e ) \\
&+
h^{1,1}_0 (x^{(1)}_0e \otimes y^{(1)}_0e )\cdot
h^{1,m-1}_0 (x^{(2)}_0e \otimes y^{(2)}_0e \otimes \cdots \otimes y^{(m)}_0e )
+\cdots \\
&+
h^{1,m-1}_0 (x^{(1)}_0e \otimes y^{(1)}_0e \otimes 
\cdots \otimes y^{(m-1)}_0e )
\cdot
h^{1,1}_0 (x^{(2)}_0e \otimes y^{(m)}_0e ) \\
&+
h^{1,m}_0 (x^{(1)}_0e \otimes y^{(1)}_0e 
\otimes \cdots \otimes y^{(m)}_0e )\cdot x^{(2)}_0e  \\
&-
h^{1,m}_0 ((x^{(1)}_0e \cdot x^{(2)}_0e )
\otimes y^{(1)}_0e \otimes \cdots \otimes y^{(m)}_0e )
 \\
-&i_0(x^{(1)}_0)e \cdot
h^{1,m}_{01} (x^{(2)}_{01}e \otimes 
i_0(y^{(1)}_0)e \otimes \cdots \otimes i_0(y^{(m)}_0)e ) \\
&-
i_0(h^{1,1}_0 (x^{(1)}_0e \otimes y^{(1)}_0e ))\cdot
h^{1,m-1}_{01} (x^{(2)}_{01}e \otimes 
i_0(y^{(2)}_0)e \otimes \cdots \otimes i_0(y^{(m)}_0)e )+\cdots \\
&-
i_0(h^{1,m-1}_0 (x^{(1)}_0e \otimes 
y^{(1)}_0e \otimes \cdots \otimes y^{(m-1)}_0e ))\cdot
h^{1,1}_{01} (x^{(2)}_{01}e \otimes i_0(y^{(m)}_0)e ) \\
&+
i_0(h^{1,m}_0 (x^{(1)}_0e \otimes 
y^{(1)}_0e \otimes \cdots \otimes y^{(m)}_0e ))\cdot x^{(2)}_{01}e  \\
&+
h^{1,m}_{01} ((i_0(x^{(1)}_0)e \cdot x^{(2)}_{01}e + 
x^{(1)}_{01}e \cdot i_1(x^{(2)}_1)e)
\otimes i_0(y^{(1)}_0)e \otimes \cdots \otimes i_0(y^{(m)}_0)e )
 \\
-&i_0(x^{(1)}_0)e \cdot
h^{1,m-1}_{01} (x^{(2)}_{01}e \otimes 
i_0(y^{(1)}_0)e \otimes \cdots \otimes i_0(y^{(m-1)}_0)e )\cdot
y^{(m)}_{01}e \\
&-
i_0(h^{1,1}_0 (x^{(1)}_0e \otimes y^{(1)}_0e ))\cdot
h^{1,m-2}_{01} (x^{(2)}_{01}e \otimes 
i_0(y^{(2)}_0)e \otimes \cdots \otimes i_0(y^{(m-1)}_0)e )\cdot
y^{(m)}_{01}e+\cdots \\
&-
i_0(h^{1,m-2}_0 (x^{(1)}_0e \otimes 
y^{(1)}_0e \otimes \cdots \otimes y^{(m-2)}_0e ))\cdot
h^{1,1}_{01} (x^{(2)}_{01}e \otimes i_0(y^{(m-1)}_0)e )\cdot
y^{(m)}_{01}e \\
&-
i_0(h^{1,m-1}_0 (x^{(1)}_0e \otimes 
y^{(1)}_0e \otimes \cdots \otimes y^{(m-1)}_0e ))\cdot x^{(2)}_{01}e
\cdot y^{(m-1)}_{01}e  \\
&+
h^{1,m-1}_{01} ((i_0(x^{(1)}_0)e \cdot x^{(2)}_{01}e + 
x^{(1)}_{01}e \cdot i_1(x^{(2)}_1)e)
\otimes i_0(y^{(1)}_0)e \otimes \cdots \otimes i_0(y^{(m-1)}_0)e )
\cdot y^{(m)}_{01}e
 \\
-&
h^{1,m}_{01} ( 
x^{(1)}_{01}e 
\otimes i_0(y^{(1)}_0)e \otimes \cdots \otimes i_0(y^{(m)}_0)e )
\cdot i_1(x^{(2)}_1)e \\
&-
h^{1,m-1}_{01} (( 
x^{(1)}_{01}e )
\otimes i_0(y^{(1)}_0)e \otimes \cdots \otimes i_0(y^{(m-1)}_0)e )
\cdot y^{(m)}_{01}e\cdot i_1(x^{(2)}_1)e.
\end{align*}
By a simple computation and Proposition \ref{central linearity},
it is equal to
\begin{align*}
&
h^{1,m}_{01} ( 
(x^{(1)}_{01}e \cdot i_1(x^{(2)}_1)e)
\otimes i_0(y^{(1)}_0)e \otimes \cdots \otimes i_0(y^{(m)}_0)e ) \\
&+
h^{1,m-1}_{01} (( 
x^{(1)}_{01}e \cdot i_1(x^{(2)}_1)e)
\otimes i_0(y^{(1)}_0)e \otimes \cdots \otimes i_0(y^{(m-1)}_0)e )
\cdot y^{(m)}_{01}e\\
&
-
h^{1,m}_{01} ( 
x^{(1)}_{01}e 
\otimes i_0(y^{(1)}_0)e \otimes \cdots \otimes i_0(y^{(m)}_0)e )
\cdot i_1(x^{(2)}_1)e \\
&-
h^{1,m-1}_{01} (( 
x^{(1)}_{01}e )
\otimes i_0(y^{(1)}_0)e \otimes \cdots \otimes i_0(y^{(m-1)}_0)e )
\cdot y^{(m)}_{01}e\cdot i_1(x^{(2)}_1)e \\
&=0
\end{align*}
Thus we have the following theorem.
\begin{theorem}
\label{obtain homotopy shuffle}
For a strict cosimplicial graded commutative DGA, there exists
a functorial homotopy shuffle system for the total DGA.
As a consequence, we have the following homomorphism of complexes:
\begin{equation}
\label{homotopy shuffle product}
\mu:\bold B_{red}(\check{C}(I,A)\otimes,\epsilon)\otimes
\bold B_{red}(\check{C}(I,A)\otimes,\epsilon)\to
\bold B_{red}(\check{C}(I,A)\otimes,\epsilon).
\end{equation}
This map is called the homotopy shuffle product.
\end{theorem}

\subsection{Artin-Schreier sheaf and homotopy shuffle product}
We apply the construction of the last section to construct
homotopy shuffle products of the Artin-Schreier DGA's.
For an $\bold F_p$-algebra, we set 
$$
AS^*(R)=(AS^*_0(R)\oplus AS^*_1(R)\to 
AS^*_{01}(R)),
$$
where $AS^*_0(R)=R,AS^*_0(R)=R,AS^*_{01}(R)=R\oplus R$
and $i_0:AS^*_0\to AS^*_{01}$ and $i_1:AS^*_{1}\to AS^*_{01}$
is given by $i_0(x)=(x,x)$ and $i_1=(x,x^p)$.
Then the morphism of complexes $AS(R)\to AS^*(R)$ given by 
\begin{align}
\label{AS and AS* quasi-iso}
& AS^0\to AS^*_0\oplus AS^*_1:x \mapsto x\oplus x \\
\nonumber
& AS^1\to AS^*_{01}:y \mapsto 0\oplus y 
\end{align}
is a quasi-isomorphism of DGA's. 
Then the DGA $AS^*$ is a total DGA of a strict cosimplicial
graded commutative DGA. By this fact and the result of the last
section, we have a homotopy shuffle product on the bar complex
of $AS^*$ for an affine scheme. We extend this construction to 
separated irreducible schemes.
The sheaf on $X_{et}$ assocated to $AS^*$ is denoted as ${\Cal AS^*}$.

Let $\Cal W=\{W_i\to X\}_{i\in I}$ be a finite set of morphisms to $X$
indexed by $I=[n]$
For an element $\bold i=(i_0, \dots,i_l)$ in $\Cal P(I)$, we set 
$W_{\bold i}=W_{i_0}\times_X\cdots \times_X W_{i_l}$.
We consider a set $\tilde I=\{0^+,0^-,\dots, n^+,n^-\}$ and $I_i=\{i^+,i^-\}$.
We introduce a total order on 
$\tilde I$ by
$i^+<i^-$ for $i=0, \dots, n$ and
$i^-<(i+1)^+$ for $i=0, \dots, n-1$.
Let $\pi$ be the projection $\tilde I \to I$.
We define a strict cosimplicial graded commutative
DGA $AS^*(\Cal W)$ indexed by $\tilde I$
by
$$
AS^*(\Cal W)_{\tilde {\bold i}}=\begin{cases}
AS_0^*(W_{\pi(\tilde {\bold i})}) & 
\text{ if }\tilde{\bold i}\subset I^+\\
AS_1^*(W_{\pi(\tilde {\bold i})}) & 
\text{ if }\tilde{\bold i}\subset I^-\\
AS_{01}^*(W_{\pi(\tilde {\bold i})}) &
\text{ otherwise.}
\end{cases}
$$ 
Since the homomorphism (\ref{AS and AS* quasi-iso}) is a
quasi-isomorphism, we have the following proposition.
\begin{proposition}
The morphism (\ref{AS and AS* quasi-iso})
induces a homomorphism of strict cosimplicial DGA
$AS(\Cal W) \to \check{C}(\tilde I/I,AS^*(\Cal W))$ indexed by $I$.
Moreover the map 
$AS(\Cal W)_I \to \check{C}(\tilde I/I,AS^*(\Cal W))_I$
is a quasi-isomorphism of DGA.
\end{proposition}
The Cech complex $\check{C}(\tilde I,AS^*(\Cal W))$ is a
strict cosimplicial graded commutative DGA.
By the last subsubsection, we have a product on
$H^0(B_{red}(\check{C}(\tilde I,AS^*(\Cal W))))$ 
arising from homotopy shuffle product.
\begin{theorem}
The multiplication
\begin{align*}
H^0(\mu):&
H^0(B_{red}(\check{C}(\tilde I,AS^*(\Cal W)))) 
\otimes 
H^0(B_{red}(\check{C}(\tilde I,AS^*(\Cal W)))) \\
&\to
H^0(B_{red}(\check{C}(\tilde I,AS^*(\Cal W))))
\end{align*}
obtained from the homotopy shuffle product
(\ref{homotopy shuffle product})
in Theorem \ref{obtain homotopy shuffle}
coincides with the product
on $\bold F_p[[\pi_1(X,\bar x)]]^*$ via the isomorphism
in Corollary \ref{isomorphism as a coalgebra}.
As a consequence, the above product $\mu$ is commutative.
\end{theorem}
\begin{proof}
Let $U\to X$ be an etale covering of X
and we set $I=[m]$.
We consider the family $\Cal U=\{U_i\to X\}_{i\in I}$ 
of covering of $X$ indexed by $I$
by setting $U=U_i$. We have a strict cosimplicial DGA $AS^*(\Cal U)$ 
indexed by $\tilde I$.
Then we have the following quasi-isomorphisms:
$$
\lim_{\underset{\Cal U,I}\to}
\check{C}(I,\bold F_p(\Cal U)) \to
\lim_{\underset{\Cal U,I}\to}
\check{C}(I,AS(\Cal U)) \to
\lim_{\underset{\Cal U,I}\to}
\check{C}(\tilde I,AS^*(\Cal U))
$$

Since the homomorphism $\check{C}(I,\bold F_p(\Cal U)) \to
\check{C}(\tilde I,AS^*(\Cal U))$ comes from the homomorphism
of strict cosimplicial graded commutative DGA,
the homotopy shuffle product on 
$B_{red}=B_{red}(\check{C}(I,\bold F_p(\Cal U)),\epsilon)$ and that on 
$B_{red}(\check{C}(\tilde I,AS^*(\Cal U)),\epsilon)$ are compatible.
It is enough to show that the homotopy shuffle product on
$B_{red}(\check{C}(I,\bold F_p(\Cal U)),\epsilon)$
is compatible with the
tensor structure on nilpotent $\bold F_p$ smooth local systems.
Let $B_{simp}=B_{simp}(\check{C}(I,\bold F_p(\Cal U)),\epsilon)$ be
the simplicial bar complex and
we assume that $M$ comes form a $B_{simp}$
comodule.
We set $A^{\bullet}=\check{C}(I,\bold F_p(\Cal U))$.
Then $M$ has a direct sum decomposition $M=\oplus_\alpha M_\alpha$
and the comodule structure defines a map
$$
\nabla_{\alpha_0,\alpha_1}:
M_{\alpha_0} \to A^1\otimes M_{\alpha_1}=\oplus_{i_0<i_1}
\bold F_p(U_{i_0}\times_XU_{i_1})\otimes
M_{\alpha_1}
$$ 
for $\alpha_0 < \alpha_1$.
We set 
$
\nabla_{\alpha_0,\alpha_1}=
\oplus_{i_0<i_1}\nabla_{\alpha_0,\alpha_1,i_0,i_1}
$
and 
$$
R_{i_0,i_1}=id-\sum_{\alpha_0<\alpha_1}
\nabla_{\alpha_0,\alpha_1,i_0,i_1}
:M\to \bold F_p(U_{i_0}\times_XU_{i_1})\otimes M.
$$ 
By the comodule condition, we have
$R_{i_0,i_1}R_{i_1,i_2}=R_{i_0,i_2}$, the set
$\{R_{i_0,i_1}\}$ of automorphism defines a local system on $X$.
Here we used the multiplication structure
$$
\bold F_p(U_{i_0}\times_XU_{i_1})\otimes 
\bold F_p(U_{i_1}\times_XU_{i_2})\to
\bold F_p(U_{i_0}\times_XU_{i_1}\times_XU_{i_2})
$$
induced by the diagonal map. Now we consider two comodules
$M_1$ and $M_2$. Then we have two systems of automorphisms
\begin{align*}
R_{i_0,i_1}^{(1)}
:M_1\to \bold F_p(U_{i_0}\times_XU_{i_1})\otimes M_1, \\
R_{i_0,i_1}^{(2)}
:M_2\to \bold F_p(U_{i_0}\times_XU_{i_1})\otimes M_2. \\
\end{align*}
And the direct sum decompositions $M_1=M_{1,\alpha_0}$ and
$M_2=M_{2,\beta_0}$.
The tensor product of the local systems associated to 
$\{R_{i_0,i_1}^{(1)}\}$ and
$\{R_{i_1,i_2}^{(2)}\}$ is equal to that associated to 
$R^{(1)}_{i_0,i_1}\otimes R^{(2)}_{i_0,i_1}$.
Here we used the algebra strcuture on 
$\bold F_p(U_{i_0}\times_XU_{i_1})$.

On the other hand, we compute the comodule structure on the
tensor product $M_1$ and $M_2$ using the homotopy shuffle product
by the following diagram:
\begin{equation}
\label{comodule from twisted shuffle product}
M_1\otimes M_2 \to B_{red}\otimes M_1 \otimes B_{red}\otimes M_2=
B_{red}\otimes B_{red} \otimes M_1\otimes M_2 
\overset{\mu \otimes 1\otimes 1}\to B_{red} \otimes M_1\otimes M_2
\end{equation}
We use the map $h^{1,1}:A^1_{\alpha_0,\alpha_1}\otimes A^1_{\beta_0,\beta_1}
\to A_{\alpha_0+\beta_0,\alpha_1+\beta_1}^1$ 
which is compatible with the
homotopy shuffle system for $B_{red}$. Note that the terms $h^{i,j}$ for $i+j>2$
vanishes by the construction of homotopy shuffle product.
Therefore the $(i_0,i_1)$ part of 
(\ref{comodule from twisted shuffle product})
is equal 
to the direct sum of
$$
-\sum_{\substack{
\alpha_0+\beta_0=\gamma_0, \\
\alpha_1+\beta_1=\gamma_1}}\nabla_{\alpha_0,\alpha_1}^{(1)}\otimes
\nabla_{\beta_0,\beta_1}^{(1)}
+\nabla^{(1)}_{\gamma_0,\gamma_1}\otimes 1
+1\otimes \nabla^{(2)}_{\gamma_0,\gamma_1}
$$
Here we used the algebra structure (commutative)
of $\bold F_p(U_{i_0}\times_X U_{i_1})$.
Therefore the comodule structure defined by the map (\ref{comodule from
twisted shuffle product}) defines
the tensor of two local systems.
\end{proof}
We expect the following.
\begin{conjecture}
\label{homotopy commutative conjecture}
More strongly, it is expected that
the homotopy shuffle product (\ref{homotopy shuffle product})
is homotopy commutative for any strict simplicial graded commutative
DGA $A^\bullet$ indexed by $I$. If it is homotopy commutative,
the product on cohomology
\begin{align*}
H^\bullet(\mu):&
H^\bullet(B_{red}(\check{C}(I,A^\bullet))) 
\otimes 
H^\bullet(B_{red}(\check{C}(I,A^\bullet))) \\
&\to
H^\bullet(B_{red}(\check{C}(I,A^\bullet)))
\end{align*}
of homotopy shuffle product is graded commutative.
The author does not know whether the homotopy shuffle 
product is homotopy commutative nor the commutativity of
 $H^{\bullet}(\mu)$ 
for any strictly simplicial
graded commutative DGA. If the base field is of characteristic zero,
there is a construction of commutative DGA out of strict simplicial
graded commutative DAG due to Thom-Whitney (see \cite{A}).
\end{conjecture}

\section{Descent theroy up to homotopy}

In this section, we treat $\bold F_q$-schemes  
with an $\bold F_q$-valued point $x$, where $q=p^e$.
We recover $\bold F_p$-Malcev completion
from $\bold F_q$-Malcev completion by the descent theory.

Let $R$ be an algebra over $\bold F_q$. We define q-Artin-Schreier
DGA $AS(R)^{(e)}$ by the total complex of the relation daigram
$$
R \begin{matrix}\overset{id}\longrightarrow \\
\underset{F^e}\longrightarrow \end{matrix}
R.
$$
Let $X$ be a separated scheme of finite type over $\bold F_q$
and $\Cal W=\{W_i\}_{i\in I}$ be a finite affine open covering
of $X$. Then the correspondence $J\mapsto AS(W_J)^{(e)}$
defines a strict cosimplicial DGA indexed by $I$, 
which is denoted as $AS(\Cal W)^{(e)}$.
The Cech complex is denoted as $\check{C}(I,AS(\Cal W)^{(e)})$.
Since the Frobenius map $F:R\to R:x\mapsto x^p$
is a ring homomorphim, we have a $F$-linear complex homomorphism
$F:AS(R) \to AS(R)$. The $e$-th power $F^e$ of $F$ is $\bold F_q$-linear
and it is homotopic to the identity map by the homotopy
$h_{Spec(R)}:AS^1(R)^{(e)}\to AS^0(R)^{(e)}:x\mapsto x$.
Since the Frobenius map is functorial for 
$\bold F_q$-algebras, it induces a $F$-linear homomorphism of DGA's
$$
\varphi:\check{C}(I,AS(\Cal W)^{(e)})\to \check{C}(I,AS(\Cal W)^{(e)}).
$$
Since the homotopy $h$ is functorial on $\bold F_q$-algebras,
$$
\prod_{\mid J\mid=p}h_{W_J}:\prod_{\mid J\mid =p}AS^1(W_J)\to
\prod_{\mid J\mid =p}AS^0(W_J)
$$ 
defines a null homotopy of 
$$
F^e-id:\check{C}(I,AS(\Cal W)^{(e)})\to \check{C}(I,AS(\Cal W)^{(e)}).
$$ 

Assume that we have an $\bold F_q$-valued point on $X$.
Then we have an augmentation map 
$\epsilon:\check{C}(I,AS(\Cal W)^{(e)})\to \bold F_q$.
Since $\check{C}(I,AS(\Cal W))$ is a DGA over $\bold F_q$,
we can define the bar complex
$B(\check{C}(I,AS(\Cal W)^{(e)}),\epsilon)_{\bold F_q}$ over $\bold F_q$.
The homomorphism $\varphi$ is $F$-linear. Therefore
we have the following $F$-linear homomorphism.
$$
\varphi:H^0(B(\check{C}(I,AS(\Cal W)),\epsilon))_{\bold F_q} \to
H^0(B(\check{C}(I,AS(\Cal W)),\epsilon))_{\bold F_q}.
$$
Similarly as in \S \ref{AS DGA}, the degree zero cohomology of the bar complex of 
$\check{C}(AS(\Cal W)^{(e)})$ over $\bold F_q$
is isomorphic to $\bold F_q$-completion of
$\pi_1(X,\bar x)$. To recover the $\bold F_p$-completion of
$\pi_1(X, \bar x)$, we need the descent data for 
$H^0(B(AS(\Cal W/X)^{(e)},\epsilon))_{\bold F_q}$. In this section,
we construct the descent data
for $H^0(B(AS(\Cal W/X)^{(e)},\epsilon))_{\bold F_q}$.

\begin{proposition}
\begin{enumerate}
\item
$\varphi^e$ is an identity.
As a consequence, there is a unique $\bold F_p$ subspace 
$H^0(B(\check{C}(I,AS(\Cal W)),\epsilon))_0$ 
called the standard $\bold F_p$ structure
in
$H^0(B(\check{C}(I,AS(\Cal W)),\epsilon))_{\bold F_q}$
such that 
$$
H^0(B(\check{C}(I,AS(\Cal W)),\epsilon))_{\bold F_q}
=\bold F_q\otimes_{\bold F_p} 
H^0(B(\check{C}(I,AS(\Cal W)),\epsilon))_0
$$
and the action of $\varphi$ is equal to $F\otimes id$ under this
identification.
\item
The standard $\bold F_p$ structure is stable under multiplication and
comultiplication.
\item
$H^0(B(\check{C}(I,AS(\Cal W)),\epsilon))_0$ 
is isomorphic to the dual of
$\bold F_p[[\pi_1(X,\bar x)]]$.
\end{enumerate}
\end{proposition}
\begin{proof}
We construct a null homotopy of 
$$
F^e-id:B(\check{C}(I,AS(\Cal W)),\epsilon)_{\bold F_q}\to 
B(\check{C}(I,AS(\Cal W)),\epsilon)_{\bold F_q}.
$$
In the proof, the tensor product ``$\otimes$'' means a tensor product 
over $\bold F_q$ unless otherwise mentioned.
Let $A$ be the Artin-Schreier DGA $\check{C}(I,AS(\Cal W))$.
For $\alpha_0<\cdots <\alpha_n$, we define a $\bold F_p$-linear map 
$$
H_{\alpha_0,\dots,\alpha_n}:\bold F_q\overset{\alpha_0}\otimes A
\overset{\alpha_1}\otimes\cdots 
\overset{\alpha_{n-1}}\otimes A\overset{\alpha_n}\otimes \bold F_q[1]
\to
\bold F_q\overset{\alpha_0}\otimes A
\overset{\alpha_1}\otimes\cdots 
\overset{\alpha_{n-1}}\otimes A\overset{\alpha_n}\otimes \bold F_q
$$  
by 
\begin{align*}
H_{\alpha_0,\dots,\alpha_n}=&1_{\bold F_q}\overset{\alpha_0}\otimes 
h_A\overset{\alpha_1}\otimes F^e_A\overset{\alpha_2}\otimes 
\cdots \overset{\alpha_{n-1}}\otimes F^e_A
\overset{\alpha_{n}}\otimes 1_{\bold F_q}\\
&+
1_{\bold F_q}\overset{\alpha_0}\otimes 
1_A\overset{\alpha_1}\otimes h_A\overset{\alpha_2}\otimes F^e_A\otimes 
\cdots \overset{\alpha_{n-1}}\otimes F^e_A\overset{\alpha_{n-1}}\otimes
1_{\bold F_q}
+\cdots \\
&+
1_{\bold F_q}\overset{\alpha_0}\otimes 
1_A\overset{\alpha_1}\otimes  
\cdots \overset{\alpha_2}\otimes 1_A
\overset{\alpha_{n-1}}\otimes h_A 
\overset{\alpha_{n-1}}\otimes 1_{\bold F_q}.
\end{align*}
Then $H=\sum_{\alpha_0,\dots, \alpha_n}H_{\alpha_0,\dots, \alpha_n}$
defines a homotopy of $F^e-id$ on 
$B(\check{C}(I,AS(\Cal W)),\epsilon)_{\bold F_q}$.
To show that $H$ defines a homotopy of $1-F^e$ on 
$B(\check{C}(I,AS(\Cal W)),\epsilon)_{\bold F_q}$, it is enough to
show that the following diagram commutes:
$$
\begin{matrix}
A\otimes A & \overset{\mu}\to & A \\
\hskip -1.0in 1\otimes h_A +h_A\otimes F^e\downarrow & & \downarrow h_A \\
A\otimes A[-1] & \overset{\mu}\to & A[-1] \\
\end{matrix}
$$
Let $(x)_{J,i}$ be an element $x\in R_J$ considered in $AS^i(R_J)$
and $J=(j_0, \dots, j_k),K=(j_k, \dots, j_{k+l}),
L=(j_0, \dots, j_{k+l})$.
Then we have
\begin{align*}
&h_A\circ \mu(((x_0)_{J,0}+(x_1)_{J,1})\otimes ((y_0)_{K,0}+(y_1)_{K,1})) =
(x_0\cdot y_1+x_1\cdot y_0^q)_{L,0} 
\end{align*}
and
\begin{align*}
&\mu\circ (1\otimes h_A-h_A\otimes (F^e))
(((x_0)_{J,0}+(x_1)_{J,1})\otimes ((y_0)_{K,0}+(y_1)_{K,1}))  \\
=&\mu(
((x_0)_{J,0}+(x_1)_{J,1})\otimes (y_1)_{K,0})-\mu
((x_1)_{J,0}\otimes ((y_0^q)_{K,0}+(y_1^q)_{K,1})) \\
=& (x_0\cdot y_1)_{L,0}+ (x_1\cdot y_1^q)_{L,1}
+(x_1\cdot y_0^q)_{L,0}-(x_1\cdot y_1^q)_{L,1}.
\end{align*}
Therefore $H$ defines a homotopy of $F^e-id$.
\end{proof}


\section{Appendix Remarks on pro $p$ completion}
\label{appendix}
In this section, we recall that $\bold F_p$ completion of a group is isomorphic
to its pro-$p$ completion.
\begin{proposition}
Let $G$ be a finit $p$ group. Then $\bold F_p[G]$ is a Artin local
field. Especially, Let $\epsilon:\bold F_p[G]\to \bold F_p$
be the augmentation map and $I=\Ker(\epsilon)$ be the augmentation 
ideal. Then $I^n=0$ for sufficiently large $n$.(See also \cite{S}, \S 15.6)
\end{proposition}
\begin{proof}
We prove the theorem by the induction of the length of upper central
series. If $G=\bold Z/q_1\bold Z\times\cdots \times \bold Z/q_n\bold Z$ 
is abelian where $q_1, \dots, q_n$ is a power of $p$, then
$\bold F[G]\simeq \bold F_p[x_1,\dots, x_n]/(x^{q_1}_1-1,\dots,x^{q_n}_n-1)$
is a local ring whose maximal ideal is the augmentation ideal $I$.
Let $G$ is a $p$ group and $P$ be a center of $G$. Then $P$ is
non-trivial. Since the augmentation ideal $I$ is generated by
$g-1$, it is enough to prove that there exists an $N$ such that
$(g_1-1)\cdots (g_N-1)=0$ for any $g_1, \dots, g_N\in G$. 
By the induction of the length of upper central series,
there exisits a positive integer $M$ such that $(g_1-1)\cdots (g_M-1)$
is contained in $Ker(\bold F_p[G]\to \bold F_p[G/P])\cap I$
for any $g_1,\dots g_M\in G$.
Therefore $(g_1-1)\cdots (g_M-1)$ can be expressed as a sum of 
$\gamma(p-q)$,
where $p,q \in P$. 
Moreover there exists $K$ such that 
\begin{align*}
&\gamma_1(p_1-q_1)\gamma_2(p_2-q_2)\cdots \gamma_K(p_K-q_K) \\
=& 
\gamma_1\gamma_2\cdots \gamma_K(p_1-q_1)(p_2-q_2)\cdots (p_K-q_K) \\
=&0 
\end{align*}
for any $\gamma_i \in G, p_i,q_i \in P$. Therefore $I^{MK}=0$.
This proves the theorem.
\end{proof}
Thus we have the following corollary.
\begin{corollary}
Let $G$ be a $p$-group.
\begin{enumerate}
\item
Then the group like element in $\bold F_p[G]$ is equal to $G$.
\item
$G$ is equal to the $\bold F_p$ Malcev completion of $G$.
\item
The pro-$p$ completion of $G$ is equal to $\bold F_p$ Malcev completion.
\end{enumerate}
\end{corollary}
\begin{proof}
(1) Let $a=\sum_{g\in G}a_g[g]$ be a group like element
in $\bold F_p[G]$. Then by the equality
$\Delta(a)=a\otimes a$, we have
$$
a_ga_h=\begin{cases}
a_g &\text{ if } g=h \\
0 &\text{ if } g\neq h \\
\end{cases}
$$
Since $a\neq 0$, there is an element $g\in G$ such that $a_g\neq 0$.
By the above equality, we have $a_g=1$. Therefore $a_h=0$ for $h\neq g$.
Thus we have $a=[g]$.
The second and the third statement follows from 
the first by the last proposition.

\end{proof}

\end{document}